\documentclass[11pt,reqno]{amsart}

\usepackage[T1]{fontenc}
\usepackage{lmodern}
\usepackage{microtype}
\usepackage{amsmath,amssymb,mathtools}
\usepackage{amsthm}
\usepackage{enumitem}
\usepackage[colorlinks=true,linkcolor=blue,citecolor=blue,urlcolor=blue]{hyperref}
\usepackage[nameinlink,capitalise,noabbrev]{cleveref}
\usepackage{comment}
\numberwithin{equation}{section}

\newtheorem{theorem}{Theorem}[section]
\newtheorem{proposition}[theorem]{Proposition}
\newtheorem{lemma}[theorem]{Lemma}
\newtheorem{corollary}[theorem]{Corollary}

\theoremstyle{definition}
\newtheorem{definition}[theorem]{Definition}

\newcommand{\R}{\mathbb R}
\newcommand{\Z}{\mathbb Z}

\newcommand{\PP}{\mathbb P}
\newcommand{\cH}{\mathcal H}
\newcommand{\cU}{\mathcal U}
\newcommand{\cF}{\mathcal F}
\newcommand{\cG}{\mathcal G}
\newcommand{\cX}{\mathcal X}
\newcommand{\SL}{\operatorname{SL}}
\newcommand{\GL}{\operatorname{GL}}
\newcommand{\Gr}{\operatorname{Gr}}
\newcommand{\Span}{\operatorname{span}}
\newcommand{\Int}{\operatorname{Int}}
\newcommand{\Opp}{\operatorname{Opp}}

\newcommand{\transp}{\mathsf T}
\newcommand{\eps}{\varepsilon}

\title{Kazhdan constants for two-element generating sets
of $\mathrm{SL}_n(\mathbb Z)$}

\author{Jvbin Yao}
\address{Research Center for Operator Algebras,
	School of Mathematical Sciences,
	East China Normal University,
	Shanghai 200241, PR China}
\email{52285500009@stu.ecnu.edu.cn}

\begin{document}

\begin{abstract}
For $n\geq3$, the group $\SL_n(\Z)$ has property $(T)$, 
so every finite generating set has a positive Kazhdan constant. 
However, it was recently shown that the infimum of these 
constants over all finite generating sets is $0$. 
This naturally raises the question of whether the infimum remains zero
when the cardinality of the generating sets is bounded in advance.
We prove that the infimum of the Kazhdan constants of
$\SL_n(\Z)$ over all two-element generating sets is $0$
for every $n\geq3$.
\end{abstract}

\maketitle

\section{Introduction}

Let $\Gamma$ be a finitely generated group and let $S\subseteq\Gamma$ be finite.
Throughout the paper, generating sets are not assumed to be symmetric.  
The Kazhdan constant relative to $S$ is
\begin{equation}\label{eq:kazhdan-constant}
 \kappa(\Gamma,S)
 =\inf_{\substack{\pi:\Gamma\to\cU(\cH)\\ \cH^\Gamma=\{0\}}}
   \inf_{\substack{\xi\in\cH\\ \|\xi\|=1}}
   \max_{s\in S}\|\pi(s)\xi-\xi\|.
\end{equation}
Property $(T)$ says that $\kappa(\Gamma,S)>0$ for every finite generating set $S$; see \cite[Remarks~1.1.4 and~1.2.2]{BHV2008}. However, $\kappa(\Gamma,S)$ depends on the choice of the generating set $S$. 
It is natural to ask whether the positivity can be made uniform over all generating sets. 
This leads to the uniform Kazhdan constant:
\[
\kappa_{\mathrm u}(\Gamma)
=
\inf_{\substack{\langle S\rangle=\Gamma\\ |S|<\infty}}
\kappa(\Gamma,S).
\]
The question whether $\kappa_{\mathrm u}(\Gamma)$ is positive has been
studied extensively; see, for example, \cite{GelanderZuk,Osin,OsinSonkin,Arzhantseva,LubotzkyYao}.
A more refined question is to ask what happens when the number of generators is bounded. For any integer $d$ greater than or equal to the minimal number of generators of $\Gamma$, define
\[
\kappa_{\le d}(\Gamma)
=
\inf_{\substack{\langle S\rangle=\Gamma\\ |S|\le d}}
\kappa(\Gamma,S).
\]
Then
$
\kappa_{\mathrm u}(\Gamma)
=
\inf_d \kappa_{\le d}(\Gamma)
=
\lim_{d\to\infty}\kappa_{\le d}(\Gamma)$.
Therefore, even if $\kappa_{\mathrm u}(\Gamma)=0$, it does not follow that there exists a fixed $d$ such that
$\kappa_{\le d}(\Gamma)=0$.
It is still possible that $\kappa_{\le d}(\Gamma)>0$
for every fixed $d$, while $\kappa_{\le d}(\Gamma)\to 0$
as $d\to\infty$.
Another important reason to consider this question is its connection with
expansion; see \cite{Arzhantseva,LubotzkyExpanders}. For symmetric generating sets, 
bounding their size ensures that the associated Cayley graphs have uniformly bounded degree.
Moreover, for generating sets of uniformly bounded size, a uniform
positive lower bound on the Kazhdan constants yields a uniform spectral
gap for the corresponding Cayley graphs of finite quotients.
The bounded-size problem is therefore naturally related to the study
of bounded-degree expanders.

Kazhdan constants for generating sets of bounded size have been
studied in several related settings.
Gelander and \.{Z}uk showed that for certain groups with property $(T)$ 
the Kazhdan constants tend to zero along generating sets of
uniformly bounded cardinality \cite{GelanderZuk}.
For finite quotients, Kassabov and Riley asked about the maximal
and minimal Kazhdan constants of $\SL_n(\mathbb Z/k\mathbb Z)$
over generating sets of bounded size \cite{KassabovRiley}.
In a different direction, Hadad proved that one can choose
generating sets $S_n$ of $\SL_n(\mathbb Z)$ whose cardinalities
are uniformly bounded and whose Kazhdan constants are bounded
below by a positive constant independent of $n\geq3$
\cite{Hadad}.
These results do not determine
$\kappa_{\le d}(\SL_n(\mathbb Z))$
for fixed $n\geq3$ and $d\geq2$.

Our previous work \cite{LubotzkyYao} established the vanishing of
$\kappa_{\mathrm u}(\Gamma)$ for every infinite finitely generated
linear group $\Gamma$.
The present paper strengthens this conclusion for
$\SL_n(\mathbb Z)$ by showing that the Kazhdan constants can tend
to zero even when the generating sets consist of only two elements.
More precisely, our main result is the following.

\begin{theorem}\label{thm:main}
For every integer $n\geq3$,
$
  \kappa_{\leq2}\bigl(\SL_n(\Z)\bigr)=0.
$
\end{theorem}

The main idea of the proof is to construct, for arbitrarily large $m$,
a pair $\{x,t\}$ generating $\Gamma=\SL_n(\Z)$ such that
$
H_m=\langle t^jxt^{-j}:0\leq j<m\rangle
$
is free of rank $m$. Thus the conjugates must freely generate a
subgroup even though $x$ and $t$ together generate all of $\Gamma$.
We start from an explicit generating pair and vary it through
Nielsen transformations, which preserve generation of $\Gamma$.
To prove that the displayed conjugates freely generate $H_m$,
we use Anosov dynamics and a relative ping-pong argument on a
suitable flag manifold. The choice of the initial pair and the
verification of the required separation conditions are the main
geometric ingredients, and are treated separately for $n\neq4$
and $n=4$. Once freeness is established, property $(T)$ implies
that $H_m$ has infinite index, while a separate algebraic argument
gives the required distinctness of the cosets.

For $n\neq4$, we take the Jordan unipotent matrix
$U=I_n+\sum_{i=1}^{n-1}E_{i,i+1}$ and its transpose, which generate
$\SL_n(\Z)$ by \cite{GowTamburini1993}. Their commutator has eigenvalues
$\lambda,1,\ldots,1,\lambda^{-1}$, where $\lambda>1$, so its attracting
and repelling dynamics are naturally described by line--hyperplane flags.
The required separation conditions can then be verified explicitly. 
For $n=4$, the same Jordan pair generates a subgroup of index $8$ rather 
than the whole group \cite{GowTamburini1993}, so we start from a different
generating pair \cite[Theorem~4.4]{Biswas2026}. For the pair used in our 
construction, the relevant commutator is not biproximal on projective space,
but its exterior square is biproximal. We therefore work on $\Gr_2(\R^4)$ and 
establish the required separation there.

The paper is organized as follows.
\Cref{sec:corridor-estimate} establishes the quasi-regular estimate.
\Cref{sec:flag-tools,sec:nielsen-criterion} develop the dynamical
tools and a criterion for the generating-pair construction.
The required generating pairs are constructed in
\cref{sec:jordan-family,sec:rank-four}.
Finally, \cref{sec:free-corridors,sec:cosets} establish freeness,
infinite index, and distinctness of the cosets, completing the
ingredients for the proof of \cref{thm:main}.

\section{A quasi-regular estimate for Kazhdan constants}\label{sec:corridor-estimate}

We first give a criterion for obtaining small Kazhdan constants from a quasi-regular representation.

\begin{lemma}\label{lem:corridor}
Let $\Gamma=\langle x,t\rangle$, let $m\geq1$, and set
$
 H=\left\langle t^jxt^{-j}:0\leq j<m\right\rangle.
$
Assume that $[\Gamma:H]=\infty$ and that $H,Ht,\ldots,Ht^m$ are pairwise distinct.  Then
\[
 \kappa(\Gamma,\{x,t\})\leq\sqrt{\frac{2}{m}}.
\]
\end{lemma}

\begin{proof}
Let $\pi$ be the right quasi-regular representation on $\ell^2(H\backslash\Gamma)$, with
$\pi(g)\delta_{H\gamma}=\delta_{H\gamma g^{-1}}$.  Since $H$ has infinite index and the action on $H\backslash\Gamma$ is transitive, $\pi$ has no nonzero invariant vector.  Put
$
 \xi=\frac1{\sqrt m}\sum_{j=0}^{m-1}\delta_{Ht^j}.
$
For $0\leq j<m$,
\[
 Ht^jx^{-1}=H(t^jx^{-1}t^{-j})t^j=Ht^j,
\]
so $\pi(x)\xi=\xi$.  The vectors $\delta_{Ht^j}$, $0\leq j\leq m$, are mutually orthogonal, and hence
$
 \|\pi(t^{-1})\xi-\xi\|^2=\frac2m.
$
Since $\pi(t)$ is unitary, the same norm is obtained with $t$ in place of $t^{-1}$.  The result follows from \eqref{eq:kazhdan-constant}.
\end{proof}

\section{Flag dynamics and relative ping-pong}\label{sec:flag-tools}

This section develops the flag-dynamical tools used later to prove
the required freeness statements.

\subsection{The two flag manifolds}

For $n\geq3$, let
\[
\cF_{1,n-1}
=\{(\ell,H):\ell\subset H\subset\R^n,\ \dim\ell=1,\ \dim H=n-1\}.
\]
We write a flag as $[v,\theta]=(\R v,\ker\theta)$,
where $v\in\R^n\setminus\{0\}$ and
$\theta\in(\R^n)^*\setminus\{0\}$ satisfy $\theta(v)=0$.
The action is $g[v,\theta]=[gv,\theta\circ g^{-1}]$.

For $n=4$, we also use $\cG=\Gr_2(\R^4)$. Fix
$\Omega=e_1\wedge e_2\wedge e_3\wedge e_4$ and define a symmetric
bilinear form $\mathcal B$ on $\wedge^2\R^4$ by
\[
	\xi\wedge\eta=\mathcal B(\xi,\eta)\Omega.
\]
If $E=\operatorname{span}(u_1,u_2)\in\cG$, any nonzero multiple of
$u_1\wedge u_2$ is called a decomposable representative of $E$.
The form $\mathcal B$ is preserved by $\wedge^2g$ for every
$g\in\SL_4(\R)$.

\begin{definition}[Opposition]
	Two flags $[v,\theta],[v',\theta']\in\cF_{1,n-1}$ are said to be
	opposite, written
	$[v,\theta]\operatorname{opp}[v',\theta']$, if
	\begin{equation}\label{eq:line-hyperplane-opposition}
		\theta(v')\neq0
		\quad\text{and}\quad
		\theta'(v)\neq0.
	\end{equation}
	
	For $E,F\in\cG$, we say that $E$ and $F$ are opposite, written
	$E\operatorname{opp}F$, if $E\oplus F=\R^4$.
	For either flag manifold $\cX$, we write
	$\Opp(\xi)=\{\eta\in\cX:\eta\operatorname{opp}\xi\}$.
	Two subsets $A,B\subset\cX$ are called mutually opposite if every
	flag in $A$ is opposite to every flag in $B$.
\end{definition}

If $\mathbf e,\mathbf f$ are nonzero decomposable representatives
of $E,F\in\cG$, respectively, then
\begin{equation}\label{eq:grassmann-opposition}
	E\operatorname{opp}F
	\quad\Longleftrightarrow\quad
	\mathcal B(\mathbf e,\mathbf f)\neq0.
\end{equation}

\begin{definition}[Proximality and biproximality]
	An element $g\in\GL(V)$ is called proximal on $\PP(V)$ if
	it has a unique eigenvalue of maximal modulus and this eigenvalue
	is simple. Equivalently, its action on $\PP(V)$ has a unique
	attracting fixed point.
	
	We use the following notions of biproximality:
	\begin{itemize}
		\item An element $g\in\SL_n(\R)$ is called biproximal on
		$\cF_{1,n-1}$ if both $g$ and $g^{-1}$ are proximal on
		$\PP(\R^n)$.
		
		\item An element $g\in\SL_4(\R)$ is called biproximal on
		$\cG=\Gr_2(\R^4)$ if both $\wedge^2g$ and
		$(\wedge^2g)^{-1}$ are proximal on
		$\PP(\wedge^2\R^4)$.
	\end{itemize}
\end{definition}
 
 In the Grassmannian case, the attracting fixed points above do
indeed belong to the Pl\"ucker-embedded Grassmannian $\Gr_2(\R^4)\subset\PP(\wedge^2\R^4)$. 
Let $A=\wedge^2g$, and choose $\xi\neq0$ such that
$A\xi=\lambda\xi$ and $[\xi]$ is the attracting eigenline of $A$.
Since
$
\det A=(\det g)^3=1
$
and $A$ is proximal, its dominant eigenvalue satisfies
$|\lambda|>1$. Indeed, otherwise the product of the moduli of all
the eigenvalues of $A$ would be strictly less than $1$. Since $A$
preserves $\mathcal B$,
$
\mathcal B(\xi,\xi)
=\mathcal B(A\xi,A\xi)
=\lambda^2\mathcal B(\xi,\xi).
$
Thus $\mathcal B(\xi,\xi)=0$. In $\wedge^2\mathbb R^4$, for
$\xi\neq0$,
\[
\mathcal B(\xi,\xi)=0
\quad\Longleftrightarrow\quad
\xi\wedge\xi=0
\quad\Longleftrightarrow\quad
\xi\ \text{is decomposable}.
\]
Hence $[\xi]$ belongs to the Pl\"ucker image of
$\Gr_2(\mathbb R^4)$. Applying the same argument to $A^{-1}$
gives the corresponding conclusion for the repelling eigenline.
Consequently, the projective proximal dynamics restricts to
$\Gr_2(\mathbb R^4)$.

 For $g$ biproximal on either flag manifold $\cX$, we denote by
 $g^+$ and $g^-$ the attracting flags of $g$ and $g^{-1}$,
 respectively. As $k\to+\infty$,
 $g^k|_{\Opp(g^-)}\to g^+$ and
 $g^{-k}|_{\Opp(g^+)}\to g^-$ locally uniformly.
 
 We use the standard notion of an Anosov subgroup for a fixed flag
 type; see \cite[Section~1.8]{DeyKapovich2023}.
 We call a subgroup $\cX$-Anosov if it is Anosov for the flag type
 of $\cX$. In particular, $\cG$-Anosov and $P_2$-Anosov both mean
 Anosov for the flag type $\{2\}$.
 
 If $g\in\SL_n(\R)$ is biproximal on $\cF_{1,n-1}$, then the
 cyclic subgroup $\langle g\rangle$ is Anosov for the flag type
 $\{1,n-1\}$. If $g\in\SL_4(\R)$ is biproximal on $\Gr_2(\R^4)$,
 then the cyclic subgroup $\langle g\rangle$ is Anosov for the
 flag type $\{2\}$.

\subsection{Asymptotic dynamics of $z^Lu$}

\begin{lemma}\label{lem:regular-ray-general}
Let $\cX$ be either $\cF_{1,n-1}$ or $\cG$, and let $z$ be biproximal on $\cX$.  Put $P=z^+$ and $R=z^-$.  For $u\in\SL_n(\R)$, set
$g_L=z^Lu$, $Q=u^{-1}R$. Then, as $L\to\infty$,
\[
 g_L|_{\Opp(Q)}\longrightarrow P,
 \qquad
 g_L^{-1}|_{\Opp(P)}\longrightarrow Q
\]
locally uniformly.  If $P$ and $Q$ are opposite, then $g_L$ is biproximal on $\cX$ for all sufficiently large $L$.
\end{lemma}

\begin{proof}
	For $\cF_{1,n-1}$, write $P=[v_+,\theta_+]$ and
	$R=[v_-,\theta_-]$. Since $z$ is biproximal, the standard
	dynamics of proximal transformations
	\cite{Benoist1997}; see also
	\cite[Lemma~2.87 and Corollary~2.88]{DrutuKapovich}, gives
	\[
	z^L[v]\longrightarrow[v_+]
	\quad\text{locally uniformly on }
	\PP(\R^n)\setminus\PP(\ker\theta_-),
	\]
	and
	\[
	(z^{-L})^{\transp}[\theta]\longrightarrow[\theta_+]
	\quad\text{locally uniformly on }
	\PP((\R^n)^*)\setminus
	\PP\{\theta:\theta(v_-)=0\}.
	\]
	Hence $z^L[v,\theta]\to P$ locally uniformly whenever
	$\theta_-(v)\neq0$ and $\theta(v_-)\neq0$, that is, on
	$\Opp(R)$.
	
	Now $Q=u^{-1}R$. Since opposition is equivariant,
	$[v,\theta]\in\Opp(Q)$ if and only if
	$u[v,\theta]\in\Opp(R)$. Therefore
	$g_L[v,\theta]=z^Lu[v,\theta]\to P$ locally uniformly on
	$\Opp(Q)$. Similarly, $z^{-L}|_{\Opp(P)}\to R$ locally uniformly,
	and since $g_L^{-1}=u^{-1}z^{-L}$, we obtain
	$g_L^{-1}|_{\Opp(P)}\to u^{-1}R=Q$ locally uniformly.
	
	For $\cG$, let $\mathbf p$ and $\mathbf r$ be nonzero decomposable
	representatives of $P$ and $R$, respectively. Since $\wedge^2z$
	preserves $\mathcal B$ and $[\mathbf r]$ is the attracting point
	of $(\wedge^2z)^{-1}$, the repelling hyperplane of $\wedge^2z$ is
	$\PP\{\xi:\mathcal B(\xi,\mathbf r)=0\}$. Thus the standard
	proximal dynamics of $\wedge^2z$ gives
	$(\wedge^2z)^L[\mathbf e]\to[\mathbf p]$ locally uniformly outside
	this projective hyperplane. Restricting to the Pl\"ucker quadric
	and using \eqref{eq:grassmann-opposition}, this exceptional set
	corresponds exactly to the planes $E\in\cG$ which are not opposite
	to $R$. Hence $z^L|_{\Opp(R)}\to P$ locally uniformly. Since
	$Q=u^{-1}R$ and opposition is equivariant,
	$g_L|_{\Opp(Q)}\to P$ locally uniformly. Applying the same argument
	to $z^{-1}$ and using $g_L^{-1}=u^{-1}z^{-L}$ gives
	$g_L^{-1}|_{\Opp(P)}\to Q$ locally uniformly.
	
	It remains to prove eventual biproximality. Let $\rho$ denote the
	relevant projective representation, namely the standard
	representation on $\R^n$ for $\cF_{1,n-1}$ and $\wedge^2$ on
	$\wedge^2\R^4$ for $\cG$. Let $\lambda_+$ be the dominant
	eigenvalue of $\rho(z)$. Since $\rho(z)$ is proximal, after
	normalization we have
	\[
	\lambda_+^{-L}\rho(z)^L\longrightarrow v_+\otimes\varphi_+,
	\qquad
	\lambda_+^{-L}\rho(z^Lu)
	\longrightarrow v_+\otimes(\varphi_+\rho(u)),
	\]
	where $[v_+]$ is the attracting point and $\ker\varphi_+$ is the
	repelling hyperplane. The rank-one operator on the right has
	nonzero eigenvalue $(\varphi_+\rho(u))(v_+)$. For
	$\cF_{1,n-1}$, this quantity is, up to a nonzero scalar,
	$\theta_-(uv_+)$; for $\cG$, it is, up to a nonzero scalar,
	$\mathcal B((\wedge^2u)\mathbf p,\mathbf r)
	=\mathcal B(\mathbf p,(\wedge^2u^{-1})\mathbf r)$. In either case
	it is nonzero because $P\operatorname{opp}Q$, where $Q=u^{-1}R$.
	
	Thus the limiting operator has one simple nonzero eigenvalue and
	all its remaining eigenvalues are zero. By continuity of the roots
	of the characteristic polynomial, $\rho(z^Lu)$ has a unique
	eigenvalue of maximal modulus for all sufficiently large $L$.
	Hence $z^Lu$ is proximal in the relevant projective representation.
	
	For the inverse, let $\mu_+$ be the dominant eigenvalue of
	$\rho(z^{-1})$. After normalization,
	\[
	\mu_+^{-L}\rho(z^{-L})\longrightarrow v_-\otimes\varphi_-,
	\qquad
	\mu_+^{-L}\rho(g_L^{-1})
	\longrightarrow \rho(u^{-1})v_-\otimes\varphi_-,
	\]
	where $[v_-]$ is the attracting point of $\rho(z^{-1})$ and
	$\ker\varphi_-$ is its repelling hyperplane. The unique possibly
	nonzero eigenvalue of the limiting rank-one operator is
	$\varphi_-(\rho(u^{-1})v_-)$. For $\cF_{1,n-1}$, this is, up to a
	nonzero scalar, $\theta_+(u^{-1}v_-)$; for $\cG$, it is, up to a
	nonzero scalar,
	$\mathcal B((\wedge^2u^{-1})\mathbf r,\mathbf p)$. Again, it is
	nonzero because $P\operatorname{opp}Q$.
	
	The same continuity argument shows that $g_L^{-1}$ is proximal for
	all sufficiently large $L$. Therefore $g_L$ is biproximal.
\end{proof}

\subsection{Anosov combination and relative ping-pong}

For either choice of $\cX$, let
\[
G=
\begin{cases}
\SL_n(\R), & \text{if }\cX=\cF_{1,n-1},\\
\SL_4(\R), & \text{if }\cX=\cG.
\end{cases}
\]
In \cref{prop:admissible-criterion}, we use the following
specialization of the Klein--Maskit combination theorem of
Dey and Kapovich \cite[Theorem~A]{DeyKapovich2023}.

\begin{theorem}[Dey--Kapovich]\label{thm:DK}
Let $A,B\subset\cX$ be compact subsets with nonempty interiors,
and assume that $A$ and $B$ are mutually opposite. Let
$\Gamma_A,\Gamma_B<G$ be $\cX$-Anosov subgroups such that
\[
\alpha B\subset\Int(A)\quad(1\neq\alpha\in\Gamma_A),
\qquad
\beta A\subset\Int(B)\quad(1\neq\beta\in\Gamma_B).
\]
Then $\langle\Gamma_A,\Gamma_B\rangle
\cong\Gamma_A*\Gamma_B$, and $\langle\Gamma_A,\Gamma_B\rangle$ is $\cX$-Anosov.
\end{theorem}

We shall also need a relative version of subgroup ping-pong.
 A $\cX$-Anosov subgroup $F<G$ is word-hyperbolic and admits
 a continuous $F$-equivariant embedding $\zeta:\partial F\to\cX$
 such that $\zeta(x)$ and $\zeta(y)$ are opposite whenever $x\neq y$.
 Moreover, $\zeta$ is dynamics preserving: if $g\in F$ has infinite
 order and $g_F^+,g_F^-\in\partial F$ are its attracting and repelling
 fixed points, then $\zeta(g_F^+)=g^+$ and $\zeta(g_F^-)=g^-$.

 Moreover, let $(h_r)$ be an escaping sequence in $F$, that is, a
 sequence leaving every finite subset of $F$. After passing to a
 subsequence, which we do not relabel, there exist
 $x_+,x_-\in\partial F$ such that
 $h_r\to x_+$ and $h_r^{-1}\to x_-$ in the Gromov compactification
 $F\cup\partial F$. For this subsequence,
 \[
 h_r|_{\Opp(\zeta(x_-))}
 \longrightarrow
 \zeta(x_+)
 \]
 locally uniformly; see
 \cite[Section~1.8 and Proposition~1.6]{DeyKapovich2023}.

We will use the following relative ping-pong criterion in
\cref{prop:free-product-corridor}.

\begin{proposition}\label{prop:relative-pingpong}
	Let $F<G$ be a non-elementary $\cX$-Anosov subgroup with boundary map
	$\zeta:\partial F\to\cX$, and let $J<F$ be a non-elementary
	quasiconvex subgroup. Let $p,q\in\partial F\setminus\partial J$ and assume
	\begin{equation}\label{eq:relative-orbits}
		\operatorname{Stab}_J(p)=\operatorname{Stab}_J(q)=\{1\},
		\qquad
		Jp\cap Jq=\varnothing.
	\end{equation}
	Let $(g_N)$ be a sequence in $G$ such that
	\begin{align}
		g_N|_{\operatorname{Opp}(\zeta(q))}
		&\longrightarrow\zeta(p),\label{eq:relative-plus}\\
		g_N^{-1}|_{\operatorname{Opp}(\zeta(p))}
		&\longrightarrow\zeta(q).\label{eq:relative-minus}
	\end{align}
	locally uniformly. Then, for all sufficiently large $N$,
	\[
	\langle J,g_N\rangle\cong J*\langle g_N\rangle.
	\]
	In particular, $g_N$ has infinite order for all sufficiently large $N$.
\end{proposition}

\begin{proof}

	Identify $\partial J$ with the limit set of $J$ in $\partial F$,
    and put $\Omega_J=\partial F\setminus\partial J$. The action of
    $J$ on $\Omega_J$ is properly discontinuous. Indeed, otherwise
    there would exist a compact set $K\subset\Omega_J$ and a sequence
    of distinct elements $h_r\in J$ such that
    $h_rK\cap K\neq\varnothing$ for every $r$.
    After passing to a subsequence, we have
    $h_r\to x_+$ and $h_r^{-1}\to x_-$ in the Gromov compactification
    $F\cup\partial F$, with $x_\pm\in\partial J$. Since
    $\zeta(K)\subset\Opp(\zeta(x_-))$
    is compact, Anosov convergence gives
    $h_r|_{\zeta(K)}\longrightarrow\zeta(x_+)$
    uniformly. Since $\zeta(x_+)\notin\zeta(K)$ and $\zeta(K)$ is
    compact, there exists a neighborhood $U$ of $\zeta(x_+)$ disjoint
    from $\zeta(K)$. For all sufficiently large $r$,
    $h_r\zeta(K)\subset U$,
    contradicting $h_r\zeta(K)\cap\zeta(K)\neq\varnothing$.
    Consequently, every $J$-orbit in $\Omega_J$ is closed and discrete
    there, and all its accumulation points in $\partial F$ lie in
    $\partial J$.

	Set
	\[
	P_0=\{p,q\},\qquad
	S=(Jp\cup Jq)\setminus P_0,\qquad
	C_0=\partial J\cup\overline S.
	\]
	The stabilizer assumptions exclude $p$ and $q$ from the closures of
	their own punctured orbits, while $Jp\cap Jq=\varnothing$ excludes the
	cross terms. Hence $C_0$ is compact and disjoint from $P_0$. Since
	$\zeta$ sends distinct boundary points to opposite flags, every flag
	in $\zeta(C_0)$ is opposite every flag in $\zeta(P_0)$. By openness of
	opposition, we may choose compact neighborhoods $A,B_p,B_q$ such that
	$\zeta(C_0)\subset\operatorname{Int}(A)$,
	$\zeta(p)\in\operatorname{Int}(B_p)$, and
	$\zeta(q)\in\operatorname{Int}(B_q)$, and, writing
	$B=B_p\cup B_q$, such that $A$ and $B$ are mutually opposite and
	\begin{equation}\label{eq:relative-neighborhoods}
		B_p\subset\operatorname{Opp}(\zeta(q)),
		\qquad
		B_q\subset\operatorname{Opp}(\zeta(p)).
	\end{equation}
	In particular, $B$ is opposite every flag in $\zeta(\partial J)$.
	
	We claim that, after shrinking $B_p$ and $B_q$ around $\zeta(p)$ and
	$\zeta(q)$,
	\begin{equation}\label{eq:J-trapping}
		hB\subset\operatorname{Int}(A)
		\qquad(h\in J\setminus\{1\}).
	\end{equation}
	Suppose otherwise. Choose nested compact neighborhoods
	$B_p^{(r)}\subset B_p$ and $B_q^{(r)}\subset B_q$ with intersections
	$\{\zeta(p)\}$ and $\{\zeta(q)\}$, respectively, and put
	$B^{(r)}=B_p^{(r)}\cup B_q^{(r)}$. Then there exist
	$h_r\in J\setminus\{1\}$ and $z_r\in B^{(r)}$ such that
	$h_rz_r\notin\operatorname{Int}(A)$.
	
	If $(h_r)$ has a constant subsequence, say $h_r=h$, then, after further
	extraction, $z_r\to\zeta(s)$ for some $s\in\{p,q\}$. By
	\eqref{eq:relative-orbits}, $hs\notin\{p,q\}$, so $hs\in S\subset C_0$.
	Hence $h_rz_r\to\zeta(hs)\in\zeta(C_0)\subset\operatorname{Int}(A)$,
	a contradiction.
	
	Otherwise, after passing to a subsequence, $(h_r)$ escapes to infinity.
	Since $J$ is quasiconvex, there exist $x_\pm\in\partial J$ such that
	$h_r\to x_+$ and $h_r^{-1}\to x_-$ in the hyperbolic compactification $F\cup\partial F$.
	By the convergence dynamics recalled above,
	$
	h_r|_{\Opp(\zeta(x_-))}
	\longrightarrow
	\zeta(x_+)
	$
	locally uniformly.
	Since the fixed compact set $B$ is opposite every flag in
	$\zeta(\partial J)$, we have $B\subset\operatorname{Opp}(\zeta(x_-))$.
	Thus $h_rz_r\to\zeta(x_+)\in\zeta(\partial J)\subset
	\operatorname{Int}(A)$, again a contradiction. This proves
	\eqref{eq:J-trapping}.
	
	Rename the resulting smaller neighborhoods as $B_p,B_q$. Since $A$ is
	opposite to both $\zeta(p)$ and $\zeta(q)$, \eqref{eq:relative-neighborhoods}
	gives
	\[
	A\cup B_p\subset\operatorname{Opp}(\zeta(q)),
	\qquad
	A\cup B_q\subset\operatorname{Opp}(\zeta(p)).
	\]
	Hence, by \eqref{eq:relative-plus}--\eqref{eq:relative-minus}, for all
	sufficiently large $N$,
	\[
	g_N(A\cup B_p)\subset\operatorname{Int}(B_p),
	\qquad
	g_N^{-1}(A\cup B_q)\subset\operatorname{Int}(B_q).
	\]
	It follows that $g_N^kA\subset\operatorname{Int}(B)$ for every
	$k\in\mathbb Z\setminus\{0\}$. If $g_N^k=1$ for some $k\neq0$, then
	$A=g_N^kA\subset\operatorname{Int}(B)$, contradicting
	$A\cap B=\varnothing$. Thus $g_N$ has infinite order.
	
	Every nontrivial element of $\langle g_N\rangle$ sends $A$ into
	$\operatorname{Int}(B)$, while every nontrivial element of $J$ sends
	$B$ into $\operatorname{Int}(A)$. The subgroup ping-pong lemma therefore
	gives
	\[
	\langle J,g_N\rangle\cong J*\langle g_N\rangle.
	\]
\end{proof}

\section{Nielsen rays and the admissibility criterion}\label{sec:nielsen-criterion}

We now construct a family of generating pairs for the same group
while keeping the commutator fixed. Under suitable flag-separation
conditions, this construction produces an Anosov free subgroup.

Let $\Gamma=\langle p,q\rangle<\SL_n(\R)$ and set
$z=pq$ and $c=pqp^{-1}q^{-1}$.
For $L\geq1$, define $a_L=z^Lp$ and $b_L=z^Lpz^{-1}$.
Then $b_L^{-1}a_L=z$, $p=z^{-L}a_L$, and $q=p^{-1}z$,
so $\langle a_L,b_L\rangle=\Gamma$.  Moreover,
\[
b_L^{-1}a_L^{-1}b_La_L
=zp^{-1}z^{-1}p
=pqp^{-1}q^{-1}=c.
\]

For convenience, we introduce the following terminology. We say that
$(p,q)$ satisfies the $\cX$-\emph{admissibility conditions} if the
following conditions hold:
\begin{enumerate}
	\item $z$ and $c$ are biproximal on $\cX$;
	\item writing $P=z^+$ and $Q=p^{-1}z^-$, the flags $P$ and $Q$
	are opposite, and each of $P,Q$ is opposite to $c^+$ and $c^-$;
	\item for every $k\in\Z\setminus\{0\}$ and every
	$R_0,S_0\in\{P,Q\}$, the flags $R_0$ and $c^kS_0$ are opposite.
\end{enumerate}
A pair satisfying these conditions is called
$\cX$-\emph{admissible}.

\begin{lemma}\label{lem:cyclic-domains-general}
Assume that $(p,q)$ satisfies the $\cX$-admissibility conditions.  Then there exist compact sets $A,B_P,B_Q\subset\cX$, each with nonempty interior, such that, with $B=B_P\cup B_Q$,
\begin{enumerate}[label=\textup{(\roman*)}]
 \item $P\in\Int(B_P)$ and $Q\in\Int(B_Q)$;
 \item $A$ and $B$ are mutually opposite, and $B_P$ and $B_Q$ are mutually opposite;
 \item $c^kB\subset\Int(A)$ for every $k\in\Z\setminus\{0\}$.
\end{enumerate}
\end{lemma}

\begin{proof}
	By north--south dynamics, $c^kP,c^kQ\to c^+$ as $k\to+\infty$, and
	$c^kP,c^kQ\to c^-$ as $k\to-\infty$. Hence
	\[
	\mathcal C_0=
	\{c^kP,c^kQ:k\in\Z\setminus\{0\}\}\cup\{c^+,c^-\}
	\]
	is compact and, by admissibility, mutually opposite to $\{P,Q\}$.
	Since the opposition relation is open, and since $\mathcal C_0$ is compact and mutually opposite to $\{P,Q\}$, we may choose sufficiently small neighborhoods of $\mathcal C_0$, $P$, and $Q$ so that all the required opposition relations persist. Moreover, since $P,Q\in\Opp(c^+)\cap\Opp(c^-)$, the neighborhoods of $P$ and $Q$ may be chosen inside $\Opp(c^+)\cap\Opp(c^-)$. Taking compact neighborhoods inside these open sets, choose compact sets $A,B_P,B_Q$ such that
	$\mathcal C_0\subset\Int(A)$, $P\in\Int(B_P)$, $Q\in\Int(B_Q)$, conditions \textup{(i)}--\textup{(ii)} hold, and we have $B\subset\Opp(c^+)\cap\Opp(c^-)$.

	By north--south dynamics,
	$c^kB\subset\Int(A)$ for all sufficiently large $|k|$.
	For each of the remaining finitely many $k\neq0$, we have
	$c^kP,c^kQ\in\mathcal C_0\subset\Int(A)$; hence, by continuity,
	shrinking $B_P$ and $B_Q$ if necessary gives
	$c^kB\subset\Int(A)$ for all such $k$ as well. The shrinking preserves
	the opposition properties. This proves \textup{(iii)}.
\end{proof}

\begin{proposition}\label{prop:admissible-criterion}
	
Assume that $(p,q)$ satisfies the $\cX$-admissibility conditions. Then there exists $L_0$ such that, for every $L\geq L_0$, the group
$F_L=\langle a_L,c\rangle$
is free, freely generated by $a_L,c$, and is Anosov for the flag type of $\cX$.	
	
\end{proposition}

\begin{proof}
Fix the domains from \cref{lem:cyclic-domains-general}.  Since $A$ is opposite to $B$ and $B_P$ is opposite to $B_Q$, we have $A\cup B_P\subset\Opp(Q)$, $A\cup B_Q\subset\Opp(P)$.
By \cref{lem:regular-ray-general}, $a_L=z^Lp$ is biproximal for all large $L$ and
$a_L|_{\Opp(Q)}\longrightarrow P$, $a_L^{-1}|_{\Opp(P)}\longrightarrow Q$.
Thus there exists $L_0$ such that, for every $L\geq L_0$,
\[
a_L(A\cup B_P)\subset\Int(B_P),
\qquad
a_L^{-1}(A\cup B_Q)\subset\Int(B_Q).
\] It follows that
$a_L^kA\subset\Int(B)$ for every $k\neq0$.
By admissibility, $c$ is biproximal, while $a_L$ is biproximal
by \cref{lem:regular-ray-general}. Hence the cyclic subgroups
$\langle c\rangle$ and $\langle a_L\rangle$ are $\cX$-Anosov.
Together with $c^kB\subset\Int(A)$ for every $k\neq0$, all the hypotheses of \cref{thm:DK} are satisfied, with
$\Gamma_A=\langle c\rangle$ and
$\Gamma_B=\langle a_L\rangle$. Therefore
\[
\langle a_L,c\rangle
\cong\langle a_L\rangle*\langle c\rangle
\cong\Z*\Z,
\]
and $\langle a_L,c\rangle$ is $\cX$-Anosov.

\end{proof}

\section{An explicit admissible pair in $\SL_n(\Z)$}\label{sec:jordan-family}

Fix $n\geq3$, and put
\[
 U=I_n+\sum_{i=1}^{n-1}E_{i,i+1},
 \qquad
 V=U^{\transp},
 \qquad
 Z=UV.
\]
By Gow and Tamburini \cite{GowTamburini1993}, if $n\neq4$, then
$\langle U,V\rangle=\SL_n(\Z)$.
We verify the $\cF_{1,n-1}$-admissibility conditions for $(U,V)$,
so that \cref{prop:admissible-criterion} applies to this pair.

\subsection{The spectra of $Z$ and $C$}

Set $\vartheta=\frac{\pi}{2n+1}$ and put
\begin{equation}\label{eq:x-y-vectors}
	x_i=\sin(2i\vartheta),
	\qquad
	y_i=(-1)^{i+1}\sin(i\vartheta)
	\qquad(1\leq i\leq n).
\end{equation}

\begin{lemma}\label{lem:Z-spectral-data}
	The matrix $Z=UV$ is biproximal on $\cF_{1,n-1}$. Its attracting
	and repelling flags are
	\begin{equation}\label{eq:P-R-general}
		P=[x,y^{\transp}],
		\qquad
		R=[y,x^{\transp}].
	\end{equation}
\end{lemma}

\begin{proof}
	The matrix $Z$ is the symmetric tridiagonal matrix with diagonal
	entries $2,\ldots,2,1$ and with $1$ on the two adjacent diagonals.
	Its eigenvalues are
    \[
		\alpha_j=2+2\cos(2j\vartheta)=4\cos^2(j\vartheta),
		\qquad 1\leq j\leq n,
	\]
	with corresponding eigenvectors
	$\bigl(\sin(2ij\vartheta)\bigr)_{i=1}^n$.
	The eigenvalues are positive and strictly decreasing. Hence both
	$Z$ and $Z^{-1}$ are proximal, so $Z$ is biproximal on
	$\cF_{1,n-1}$.
	
	The vectors $x$ and $y$ are eigenvectors for the largest and smallest
	eigenvalues, respectively. Since $Z$ is symmetric, the corresponding
	attracting and repelling flags are given by \eqref{eq:P-R-general}.
\end{proof}

 With $P$ and $R$ as above, set
 \begin{equation}\label{eq:Q-general}
 	Q=U^{-1}R=[q,h^{\transp}],
 	\qquad
 	q=U^{-1}y,
 	\qquad
 	h^{\transp}=x^{\transp}U.
 \end{equation}

 Let
 \[
 C=UVU^{-1}V^{-1},
 \qquad
 \eps=(-1)^{n-1},
 \]
 and set
 \[
 	\lambda=\frac{n+1+\sqrt{(n+1)^2-4}}{2},
 	\qquad
 	\nu=\lambda^{-1}.
 \]
 
 \begin{lemma}\label{lem:C-spectral-data}
 	The matrix $C$ is biproximal on $\cF_{1,n-1}$.
 \end{lemma}
 
 \begin{proof}
 	Since $(U^{-1})_{ij}=(-1)^{j-i}$ for $j\geq i$ and is zero
 	otherwise, direct multiplication gives
 	\begin{equation}\label{eq:C-rank-two}
 		C=I_n+e_1r^{\transp}+e_ns^{\transp},
 	\end{equation}
 	where
 	\[
 		r_j=(-1)^{j-1}(n+1-j),
 		\qquad
 		s_j=(-1)^{n-j+1}.
 	\]
 	
 	Put $E=\Span(e_1,e_n)$. Then $E$ is $C$-invariant, and the matrix
 	of the restriction $C|_E$ with respect to the basis $(e_1,e_n)$ is
 	\[
 		B=
 		\begin{pmatrix}
 			n+1 & \eps\\
 			-\eps & 0
 		\end{pmatrix}.
 	\]
 	
 	Now set $K=\ker r^{\transp}\cap\ker s^{\transp}$. By
 	\eqref{eq:C-rank-two}, $C|_K=I_K$. Moreover, the map
 	$E\to\mathbb R^2$, $v\mapsto(r^{\transp}v,s^{\transp}v)$, is
 	invertible. Hence $K\cap E=\{0\}$ and $\dim K=n-2$, so
 	$\mathbb R^n=K\oplus E$. Therefore
 	\[
 		\chi_C(T)=(T-1)^{n-2}\bigl(T^2-(n+1)T+1\bigr).
 	\]
 	
 	Thus the eigenvalues of $C$ are $\lambda,1,\ldots,1,\nu$, where $1$
 	has multiplicity $n-2$ and $\lambda>1>\nu>0$. In particular, $C$
 	has a unique eigenvalue of maximal modulus and a unique eigenvalue of
 	minimal modulus. Hence both $C$ and $C^{-1}$ are proximal, so $C$ is
 	biproximal on $\cF_{1,n-1}$.
 \end{proof}

\subsection{The opposition conditions}

It remains to verify admissibility conditions (2)--(3).
We reduce these opposition conditions to the nonvanishing of certain
scalar pairings, which we verify directly.

For $R_0,S_0\in\{P,Q\}$, use the representatives
$(v_P,\theta_P)=(x,y^{\transp})$ and
$(v_Q,\theta_Q)=(q,h^{\transp})$ from
\eqref{eq:P-R-general} and \eqref{eq:Q-general}, and define
\[
	s_{R_0S_0}(k)=\theta_{R_0}C^kv_{S_0},
	\qquad k\in\Z.
\]
Put
$c_0=\cos\vartheta$,
$A_j=(-1)^{j-1}\sin(2j\vartheta)$, and
$S_i=\sum_{j=i}^n\sin(j\vartheta)$.

\begin{lemma}\label{lem:extremal-coefficients}
For any covector $\theta$ and vector $v$, write
\[
 \theta C^kv=A_\theta(v)\lambda^k+B_\theta(v)+D_\theta(v)\nu^k.
\]
For $\mu\in\{\lambda,\nu\}$, the coefficient of $\mu^k$ is
\[
 K_\mu F_\theta(\mu)G_v(\mu),
 \qquad
 K_\mu=\frac{\mu^2}{(\mu-1)^2(\mu+1)}>0,
\]
where
\[
 F_\theta(\mu)=\theta_1-\eps\mu^{-1}\theta_n,
 \qquad
 G_v(\mu)=r^{\transp}v+\eps\mu^{-1}s^{\transp}v.
\]
\end{lemma}

\begin{proof}
	Let $K=\ker r^{\transp}\cap\ker s^{\transp}$ and
	$E=\Span(e_1,e_n)$. As shown above,
	$\R^n=K\oplus E$, $C|_K=I$, and
	$[C|_E]_{(e_1,e_n)}=B$.
	Hence every $v\in\R^n$ can be written uniquely as
	$v=v_0+ae_1+be_n$, with $v_0\in K$.
	
	Since $r^{\transp}v_0=s^{\transp}v_0=0$, we have
	\[
	\binom{r^{\transp}v}{s^{\transp}v}
	=
	\begin{pmatrix}
		r_1&r_n\\
		s_1&s_n
	\end{pmatrix}
	\binom ab
	=
	(B-I)\binom ab.
	\]
	Since the eigenvalues of $B$ are $\lambda$ and $\nu$, neither of
	which is $1$, the matrix $B-I$ is invertible. Thus
	\[
	\binom ab
	=
	(B-I)^{-1}
	\binom{r^{\transp}v}{s^{\transp}v}.
	\]
	
	Fix $\mu\in\{\lambda,\nu\}$. Since $\lambda\nu=1$, the other
	eigenvalue of $B$ is $\mu^{-1}$. A right $\mu$-eigenvector and a
	left $\mu$-eigenvector of $B$ are
	$u_\mu=(1,-\eps\mu^{-1})^{\transp}$ and
	$\ell_\mu=(1,\eps\mu^{-1})$.
	Moreover, $\ell_\mu u_\mu=1-\mu^{-2}$ and
	$\ell_\mu u_{\mu^{-1}}=0$.
	
	Since $u_\mu$ and $u_{\mu^{-1}}$ form a basis of $\R^2$, write
	$\binom ab=c_\mu u_\mu+c_{\mu^{-1}}u_{\mu^{-1}}$.
	Applying $\ell_\mu$ gives
	\[
	c_\mu
	=
	\frac{\ell_\mu\binom ab}{1-\mu^{-2}}.
	\]
	
	Since $\ell_\mu(B-I)=(\mu-1)\ell_\mu$, we have
	$\ell_\mu(B-I)^{-1}=(\mu-1)^{-1}\ell_\mu$. Therefore
	\[
	\ell_\mu\binom ab
	=
	\frac{r^{\transp}v+\eps\mu^{-1}s^{\transp}v}{\mu-1},
	\]
	and hence
	\[
	c_\mu
	=
	\frac{r^{\transp}v+\eps\mu^{-1}s^{\transp}v}
	{(1-\mu^{-2})(\mu-1)}.
	\]
	
	Since $C|_K=I$ and $C|_E$ is represented by $B$,
	\[
	C^kv
	=
	v_0
	+
	c_\mu\mu^k
	\bigl(e_1-\eps\mu^{-1}e_n\bigr)
	+
	c_{\mu^{-1}}\mu^{-k}
	\bigl(e_1-\eps\mu e_n\bigr).
	\]
	Thus, writing $\theta_i=\theta(e_i)$, the coefficient of $\mu^k$
	in $\theta C^kv$ is
	$c_\mu(\theta_1-\eps\mu^{-1}\theta_n)$.
	
	Finally, since
	$1-\mu^{-2}=(\mu-1)(\mu+1)/\mu^2$, this coefficient is
	\[
	\frac{\mu^2}{(\mu-1)^2(\mu+1)}
	\bigl(\theta_1-\eps\mu^{-1}\theta_n\bigr)
	\bigl(r^{\transp}v+\eps\mu^{-1}s^{\transp}v\bigr),
	\]
	as claimed.
\end{proof}

We now determine the signs of all extremal coefficients.

\begin{lemma}
	\label{lem:extremal-signs}
	For $R_0,S_0\in\{P,Q\}$, write
	\[
	s_{R_0S_0}(k)
	=
	A_{\theta_{R_0}}(v_{S_0})\lambda^k
	+
	B_{\theta_{R_0}}(v_{S_0})
	+
	D_{\theta_{R_0}}(v_{S_0})\nu^k.
	\]
	Then
	$A_{\theta_{R_0}}(v_{S_0})>0$
	for all $R_0,S_0\in\{P,Q\}$. Moreover, the signs of
	$D_{\theta_{R_0}}(v_{S_0})$ are given by
	\[
	\begin{array}{c|cc}
		& S_0=P & S_0=Q\\
		\hline
		R_0=P & \eps & +\\
		R_0=Q & + & \eps
	\end{array}.
	\]
\end{lemma}

\begin{proof}
	By \cref{lem:extremal-coefficients}, for
	$\mu\in\{\lambda,\nu\}$, the coefficient of $\mu^k$ in
	$\theta C^kv$ is $K_\mu F_\theta(\mu)G_v(\mu)$, where $K_\mu>0$.
	Thus it suffices to determine the signs of the corresponding
	$F$- and $G$-factors.
	
    From \eqref{eq:x-y-vectors},
	$y_1=\sin\vartheta$ and
	$y_n=\eps\cos(\vartheta/2)$. Also, with $x_0=0$,
	$h_j=x_j+x_{j-1}$, so $h_1=\sin(2\vartheta)$ and
	$h_n=2\sin(2\vartheta)c_0$. Hence
	$F_{\theta_P}(\mu)=\sin\vartheta-\mu^{-1}\cos(\vartheta/2)$ and
	$F_{\theta_Q}(\mu)=\sin(2\vartheta)(1-2\eps\mu^{-1}c_0)$.

	Since $\lambda+\lambda^{-1}=n+1$, we have $\lambda>n$. It follows that
	\begin{equation}\label{eq:F-signs}
		F_{\theta_P}(\lambda)>0,
		\quad
		F_{\theta_P}(\nu)<0,
		\quad
		F_{\theta_Q}(\lambda)>0,
		\quad
		\operatorname{sgn}F_{\theta_Q}(\nu)=-\eps.
	\end{equation}
	For the first inequality, it is enough to show
	$2\lambda\sin(\vartheta/2)>1$. Put $t=\vartheta/2$. Since
	$t<1/4$ and $\sin t>t\cos t$, we have
	\[
	2\lambda\sin t
	>
	2nt\cos t
	=
	\frac{n\pi}{2n+1}\cos t
	\geq
	\frac{3\pi}{7}\cos t.
	\]
	Moreover, $\cos t>1-t^2/2>4/5$, and hence
    $(3\pi/7)\cos t > 36/35 > 1$.
	Thus $F_{\theta_P}(\lambda)>0$.
	For the remaining signs, note that $\nu^{-1}=\lambda$.
	Since $2\sin(\vartheta/2)<1<\lambda$, we have
	$F_{\theta_P}(\nu)<0$. Also,
	$F_{\theta_Q}(\lambda)>0$ because $\sin(2\vartheta)>0$ and
	$1-2\eps\lambda^{-1}c_0>0$. Finally, since
	$2\lambda c_0>1$, the factor $1-2\eps\lambda c_0$ has sign
	$-\eps$, and hence
	$\operatorname{sgn}F_{\theta_Q}(\nu)=-\eps$.

	Next consider $G_x$. Recall
	$A_j=(-1)^{j-1}\sin(2j\vartheta)$. By the identities in
	\cref{app:sine-sums},
	\begin{equation}\label{eq:alternating-sums}
		\sum_{j=1}^n A_j=\frac{\sin\vartheta}{2c_0},
		\qquad
		\sum_{j=1}^n(n+1-j)A_j
		=\frac{\sin\vartheta\bigl(2(n+1)c_0-\eps\bigr)}{4c_0^2}.
	\end{equation}
    Since $\lambda+\lambda^{-1}=n+1$, the formula above gives
    \[
    G_x(\lambda)
    =
    \frac{\sin\vartheta}{4c_0^2}(2\lambda c_0-\eps),
    \qquad
    G_x(\nu)
    =
    \frac{\sin\vartheta}{4c_0^2}(2\nu c_0-\eps).
    \]
    As $\lambda>n\geq3$ and $\nu=\lambda^{-1}<1/3$, we obtain
    \begin{equation}\label{eq:G-P-signs}
    	G_x(\lambda)>0,
    	\qquad
    	\operatorname{sgn}G_x(\nu)=-\eps.
    \end{equation}

	Finally, recall that
	$S_i=\sum_{j=i}^n\sin(j\vartheta)$. Since
	$0<j\vartheta<\pi/2$ for $1\leq j\leq n$, we have $S_i>0$.
	The formula for $U^{-1}$ gives $q_i=(-1)^{i+1}S_i$, and hence
	\[
	G_q(\mu)
	=
	\sum_{i=1}^n(n+1-i-\mu^{-1})S_i.
	\]
	For $\mu=\lambda$, every summand is positive since
	$\lambda^{-1}<1\leq n+1-i$. For $\mu=\nu$, every summand is negative
	since $\nu^{-1}=\lambda>n\geq n+1-i$. Hence
	\begin{equation}\label{eq:G-Q-signs}
		G_q(\lambda)>0,
		\qquad
		G_q(\nu)<0.
	\end{equation}
	
	Combining \eqref{eq:F-signs}, \eqref{eq:G-P-signs}, and
	\eqref{eq:G-Q-signs} with $K_\mu>0$, we see that all the
	$\lambda^k$-coefficients are positive, while the signs of the
	$\nu^k$-coefficients are exactly those in the asserted table.

\end{proof}

\begin{lemma}\label{lem:UV-admissibility}
For every $n\geq3$, the following statements hold for the matrices and flags above.
\begin{enumerate}[label=\textup{(\roman*)}]
 \item $s_{PQ}(k)>0$ and $s_{QP}(k)>0$ for every $k\in\Z$.
 \item $s_{PP}(k)\neq0$ and $s_{QQ}(k)\neq0$ for every $k\neq0$.
 \item $P,Q$ are opposite to each other and to both $C^+$ and $C^-$.  Moreover, for every $k\neq0$ and every $R_0,S_0\in\{P,Q\}$, the flags $R_0$ and $C^kS_0$ are opposite.
\end{enumerate}
\end{lemma}

\begin{proof}
	
	By \cref{lem:extremal-signs}, the coefficients of both $\lambda^k$
	and $\nu^k$ are positive in each of the sequences $s_{PQ}(k)$ and
	$s_{QP}(k)$. Moreover,  $s_{PQ}(0)=y^{\transp}U^{-1}y>0$ and $s_{QP}(0)=x^{\transp}Ux>0$.
	Indeed, after combining the alternating signs in $y$ and $U^{-1}$,
	every summand in $y^{\transp}U^{-1}y$ is positive. Also,
	$
	x^{\transp}Ux
	=
	\sum_{i=1}^n x_i^2+\sum_{i=1}^{n-1}x_ix_{i+1}>0,
	$
	since all coordinates of $x$ are positive.
	
	The identity $C=ZU^{-1}Z^{-1}U$ gives
	$CU^{-1}=ZU^{-1}Z^{-1}$. Since $Z$ is symmetric and $y$ is an
	eigenvector of $Z$,
	\[
	s_{PQ}(1)
	=
	y^{\transp}CU^{-1}y
	=
	y^{\transp}U^{-1}y
	=
	s_{PQ}(0).
	\]
	Similarly, $C^{-1}=U^{-1}ZUZ^{-1}$, and since $Z$ is symmetric and
	$x$ is an eigenvector of $Z$,
	\[
	s_{QP}(-1)
	=
	x^{\transp}UC^{-1}x
	=
	x^{\transp}Ux
	=
	s_{QP}(0).
	\]
	
	Each of the functions $s_{PQ}(t)$ and $s_{QP}(t)$ has the form
	$A\lambda^t+B+D\lambda^{-t}$ on $\R$, with $A,D>0$, and is therefore
	strictly convex. Hence the equality
	$s_{PQ}(0)=s_{PQ}(1)$ implies that its minimum over $\Z$ is attained
	at $0$ and $1$, while
	$s_{QP}(-1)=s_{QP}(0)$ implies that its minimum over $\Z$ is attained
	at $-1$ and $0$. Since these values are positive, we obtain
	$s_{PQ}(k)>0$ and $s_{QP}(k)>0$ for every $k\in\Z$. This proves
	\textup{(i)}. In particular, $P$ and $Q$ are opposite.

	By \cref{lem:extremal-signs}, for both $s_{PP}(k)$ and $s_{QQ}(k)$,
	the coefficient of $\lambda^k$ is positive, while the coefficient of
	$\nu^k$ has sign $\eps$. Suppose first that $n$ is even, so $\eps=-1$. Write
	$
	s(k)=A\lambda^k+B+D\nu^k
	$
	for either $s_{PP}(k)$ or $s_{QQ}(k)$. Then $A>0>D$. Since
	$s(0)=0$ and $\nu=\lambda^{-1}$, we have $B=-A-D$, and hence
	\[
	s(k)=A(\lambda^k-1)+D(\lambda^{-k}-1).
	\]
	If $k>0$, then $\lambda^k-1>0$ and $\lambda^{-k}-1<0$, so both
	terms on the right are positive. Thus $s(k)>0$. Similarly, if
	$k<0$, both terms are negative, and hence $s(k)<0$.

    Assume now that $n$ is odd, so the coefficients of both $\lambda^k$
    and $\nu^k$ in $s_{PP}(k)$ and $s_{QQ}(k)$ are positive. We begin by proving that $s_{PP}(1)>0$.
    Let
    $T_0=\sum_{j=1}^n A_j$ and
    $T_1=\sum_{j=1}^n(n+1-j)A_j$.
    By \eqref{eq:alternating-sums},
    \[
    T_0=\frac{\sin\vartheta}{2c_0},
    \qquad
    T_1=\frac{\sin\vartheta(2(n+1)c_0-1)}{4c_0^2}.
    \]
    Using \eqref{eq:C-rank-two}, we obtain
    \[
    	s_{PP}(1)
    	=\sin\vartheta\,T_1-\cos\frac{\vartheta}{2}\,T_0
    	=\frac{\sin\vartheta\cos(\vartheta/2)}{2c_0^2}
    	\left[
    	\sin\frac{\vartheta}{2}\bigl(2(n+1)c_0-1\bigr)-c_0
    	\right].
    \]

The bracket is positive.  Indeed, put $t=\vartheta/2<1/4$.  Since $\sin t>t\cos t$, $c_0=\cos(2t)>7/8$, $\cos t>31/32$, and $\pi>3$,
\[
 2(n+1)c_0\sin t
 >\frac\pi2c_0\cos t
 >\frac{651}{512}
 >\frac54
 >c_0+\sin t.
\]

 A direct multiplication gives
 \begin{equation}\label{eq:C-inverse-rank-two}
 	C^{-1}=I_n+e_1(r^-)^{\transp}+e_n(s^-)^{\transp},
 	\qquad
 	r^-_j=(-1)^j,
 	\quad
 	s^-_j=(-1)^{n+j}j.
 \end{equation}
 When $n$ is odd, we have
 $(r^-)^{\transp}x=-T_0$ and
 $(s^-)^{\transp}x=\sum_{j=1}^n jA_j$. Moreover,
 \[
 \sum_{j=1}^n jA_j
 =(n+1)T_0-T_1
 =\frac{\sin\vartheta}{4c_0^2}.
 \]
 Using $y^{\transp}x=0$, $y_1=\sin\vartheta$, and
 $y_n=\cos(\vartheta/2)$, we obtain
 \[
 	s_{PP}(-1)
 	=\frac{\sin\vartheta}{4c_0^2}
 	\left(\cos\frac\vartheta2-2c_0\sin\vartheta\right)>0.
 \]
 Indeed,
 $2c_0\sin\vartheta=\sin(2\vartheta)$, while
 $\cos(\vartheta/2)=\sin(\pi/2-\vartheta/2)>\sin(2\vartheta)$
 since $\vartheta<\pi/5$.

  For $s_{QQ}$, put $W=\sum_i S_i$ and
  $R_1=\sum_i(n+1-i)S_i$. The equality $s_{PQ}(1)=s_{PQ}(0)$ proved above is equivalent to
  $\sin\vartheta\,R_1=\cos(\vartheta/2)\,W$, and hence
  $R_1=W/(2\sin(\vartheta/2))$. Using \eqref{eq:C-rank-two}, we obtain
  \[
  	s_{QQ}(1)
  	=\sin(2\vartheta)W
  	\left(\frac1{2\sin(\vartheta/2)}-2c_0\right)>0.
  \]
  Indeed,
  $4c_0\sin(\vartheta/2)<4\sin(\vartheta/2)
  <4(\vartheta/2)<1$.
  
  Next, put $J=\sum_i iS_i$. Since $S_i>0$, we have $J\ge W$.
  Using \eqref{eq:C-inverse-rank-two},
  \[
  	s_{QQ}(-1)=\sin(2\vartheta)(2c_0J-W)>0,
  \]
  because $2c_0>1$.
  
  Thus $s_{PP}(\pm1)>0$ and $s_{QQ}(\pm1)>0$. For either of these
  sequences, write
  $s(k)=A\lambda^k-(A+D)+D\lambda^{-k}$, where $A,D>0$.
  Since
  $s(1)=(\lambda-1)(A-D/\lambda)>0$, we have
  $A>D/\lambda$, and hence, for every $k\ge1$,
  \[
  s(k)=(\lambda^k-1)\left(A-\frac{D}{\lambda^k}\right)>0.
  \]
  Similarly, $s(-1)>0$ implies $s(k)>0$ for every $k\le-1$.
  This proves \textup{(ii)}.

 The opposition criterion \eqref{eq:line-hyperplane-opposition} says that
 $R_0$ is opposite to $C^kS_0$ exactly when
 $
 s_{R_0S_0}(k)s_{S_0R_0}(-k)\neq0.
 $
 Hence the last assertion in \textup{(iii)} follows from
 \textup{(i)}--\textup{(ii)}.
 Finally, for $R_0\in\{P,Q\}$, the nonvanishing of
 $F_{\theta_{R_0}}(\lambda)$ and $G_{v_{R_0}}(\nu)$ shows that
 $R_0$ is opposite to $C^+$, while the nonvanishing of
 $F_{\theta_{R_0}}(\nu)$ and $G_{v_{R_0}}(\lambda)$ shows that
 $R_0$ is opposite to $C^-$. This completes the proof.

\end{proof}

\begin{proposition}\label{prop:jordan-admissible}
For every $n\geq3$, the pair $(U,V)$ is $\cF_{1,n-1}$-admissible.  If $n\neq4$, it is a generating pair of $\SL_n(\Z)$.
\end{proposition}

\begin{proof}
 By \cref{lem:Z-spectral-data,lem:C-spectral-data}, both $Z=UV$ and
 $C$ are biproximal on $\cF_{1,n-1}$. The admissibility conditions
 \textup{(2)}--\textup{(3)} follow from
 \cref{lem:UV-admissibility}, while
$\langle U,V\rangle=\SL_n(\Z)$ for $n\neq4$ by
Gow and Tamburini \cite{GowTamburini1993}.
\end{proof}

\begin{corollary}\label{cor:jordan-free-anosov}
	For every $n\geq3$, there exists $L_0$ such that, for every
	$L\geq L_0$, the group $\langle Z^LU,C\rangle$ is free and
	$\cF_{1,n-1}$-Anosov.
\end{corollary}

\begin{proof}
	This follows immediately from
	\cref{prop:jordan-admissible,prop:admissible-criterion}.
\end{proof}

\section{The rank-four admissible pair}\label{sec:rank-four}

For $n=4$, the Jordan pair from \cref{sec:jordan-family} does not
generate $\SL_4(\Z)$. We therefore construct a different generating
pair for $\SL_4(\Z)$ and prove that it is
$\Gr_2(\R^4)$-admissible.

  \subsection{An explicit generating pair}
  
  \begin{lemma}\label{lem:rank-four-generating-pair}
  	The matrices
  	\begin{equation}\label{eq:p4-q4}
  		p=
  		\begin{pmatrix}
  			-1&-4&3&0\\
  			-3&-6&6&-1\\
  			0&1&0&0\\
  			-1&-4&4&0
  		\end{pmatrix},
  		\qquad
  		q=
  		\begin{pmatrix}
  			-1&-1&1&0\\
  			0&-1&0&0\\
  			0&0&-1&0\\
  			-1&-1&1&-1
  		\end{pmatrix}
  	\end{equation}
  	generate $\SL_4(\Z)$. Moreover, $z=pq$ is biproximal on
  	$\cG=\Gr_2(\R^4)$.
  \end{lemma}
  
  \begin{proof}
  	Let
  	\[
  	u=
  	\begin{pmatrix}
  		0&1&0&0\\
  		0&0&1&0\\
  		0&0&0&1\\
  		-1&0&0&0
  	\end{pmatrix},
  	\qquad
  	v=
  	\begin{pmatrix}
  		1&1&0&0\\
  		-1&0&0&0\\
  		0&0&1&0\\
  		0&0&0&1
  	\end{pmatrix}.
  	\]
  	By \cite[Theorem~4.4]{Biswas2026},
  	$\langle u,v\rangle=\SL_4(\Z)$.
  	Starting from $(X,Y)=(u,v)$, perform the elementary Nielsen moves
  	\[
  	\begin{aligned}
  		X&\leftarrow XY, &Y&\leftarrow XY, &X&\leftarrow XY, &Y&\leftarrow XY,\\
  		X&\leftarrow XY, &X&\leftarrow XY, &X&\leftarrow XY,
  	\end{aligned}
  	\]
  	using the current values at each step. Direct multiplication gives
  	the pair $(p,q)$ in \eqref{eq:p4-q4}. Since elementary Nielsen
  	transformations preserve the generated subgroup,
  	$\langle p,q\rangle=\SL_4(\Z)$.
  	
  	Now
  	\[
  	z=pq=
  	\begin{pmatrix}
  		1&5&-4&0\\
  		4&10&-10&1\\
  		0&-1&0&0\\
  		1&5&-5&0
  	\end{pmatrix},
  	\]
  	and
  	\begin{align}\label{eq:z4-charpoly}
  		\chi_z(T)
  		&=T^4-11T^3-25T^2-11T+1\notag\\
  		&=\left(T^2-\frac{11+\sqrt{229}}2T+1\right)
  		\left(T^2-\frac{11-\sqrt{229}}2T+1\right).
  	\end{align}
  	Let $\rho>1$ and $\sigma<-1$ be the roots outside the unit circle
  	of the first and second quadratic factors, respectively. Their
  	moduli satisfy
  	\[
  	\rho>|\sigma|>1>|\sigma|^{-1}>\rho^{-1}.
  	\]
  	Hence $\wedge^2z$ and $(\wedge^2z)^{-1}$ are proximal, so $z$ is
  	biproximal on $\cG$.
  \end{proof}

  Let $P=z^+$ and $R=z^-$. Thus $P$ is the plane spanned by the
  $\rho$- and $\sigma$-eigenvectors, while $R$ is the plane spanned by
  the $\rho^{-1}$- and $\sigma^{-1}$-eigenvectors. Set
 \[
 	c=pqp^{-1}q^{-1}=
 	\begin{pmatrix}
 		-1&0&-1&1\\
 		-1&1&0&1\\
 		0&0&1&0\\
 		-3&1&-2&2
 	\end{pmatrix}.
 \]
 
 \begin{lemma}\label{lem:c4-biproximal}
 	The matrix $c$ is biproximal on $\cG=\Gr_2(\R^4)$.
 \end{lemma}

 \begin{proof}
 	Define
 	\[
 		d=
 		\begin{pmatrix}
 			-1&1&-2&0\\
 			-2&1&-1&1\\
 			0&0&1&0\\
 			-1&2&-2&0
 		\end{pmatrix}.
 	\]
 	Direct multiplication gives
 	\begin{equation}\label{eq:d-square}
 		d^2=c,
 		\qquad
 		\chi_d(T)=(T-1)(T^3-T-1).
 	\end{equation}
 	Let $\varrho>1$ be the real root of $T^3-T-1$, and let
 	$\zeta,\bar\zeta$ be the remaining two roots. Since
 	$\varrho\zeta\bar\zeta=1$, we have
 	$|\zeta|=|\bar\zeta|=\varrho^{-1/2}$. Hence the eigenvalues of
 	$c=d^2$ have moduli
 	\[
 	\varrho^2,\qquad 1,\qquad \varrho^{-1},\qquad \varrho^{-1}.
 	\]
 	Therefore the eigenvalues of $\wedge^2c$ have moduli
 	\[
 	\varrho^2,\qquad
 	\varrho,\qquad
 	\varrho,\qquad
 	\varrho^{-1},\qquad
 	\varrho^{-1},\qquad
 	\varrho^{-2}.
 	\]
 	Thus $\wedge^2c$ has a unique eigenvalue of maximal modulus and a
 	unique eigenvalue of minimal modulus. Hence both $\wedge^2c$ and
 	$(\wedge^2c)^{-1}$ are proximal, so $c$ is biproximal on $\cG$.
 \end{proof}

\subsection{Reduction to a scalar sequence}
The opposition relations involving the translates
$c^kP$ and $c^kQ$ can all be expressed in terms of a single
scalar sequence.

\begin{lemma}\label{lem:rank-four-Phi-reduction}
	Set $Q=p^{-1}R$ and $C=\wedge^2c$. Then there exist nonzero Pl\"ucker representatives
	$\mathbf p$ and $\mathbf q$ of $P$ and $Q$, respectively, and an even
	sequence $(\Phi_j)_{j\in\Z}$ such that, for every $k\in\Z$,
	\begin{align}
		\mathcal B(\mathbf p,C^k\mathbf p)&=\Phi_{2k},
		&
		\mathcal B(\mathbf q,C^k\mathbf q)&=\Phi_{2k},
		\label{eq:self-Phi}\\
		\mathcal B(\mathbf p,C^k\mathbf q)&=-\Phi_{2k-1},
		&
		\mathcal B(\mathbf q,C^k\mathbf p)&=-\Phi_{2k+1}.
		\label{eq:cross-Phi}
	\end{align}
\end{lemma}

\begin{proof}
	Set
	\[
	s=dp^{-1}=
	\begin{pmatrix}
		6&0&5&-5\\
		15&-1&15&-10\\
		-1&0&0&1\\
		6&0&6&-5
	\end{pmatrix}.
	\]
	Direct multiplication gives $s^2=I$ and $szs=z^{-1}$. Thus
	$sR=P$. Since $Q=p^{-1}R$ and $s=dp^{-1}$, we obtain
	$dQ=P$, and hence $Q=d^{-1}P$.
	
	Put $D=\wedge^2d$. Since $c=d^2$, we have
	$C=\wedge^2c=D^2$. Choose a nonzero Pl\"ucker representative
	$\mathbf p$ of $P$ and set $\mathbf q=-D^{-1}\mathbf p$. Then
	$\mathbf q$ is a Pl\"ucker representative of $Q$. Define
	$\Phi_j=\mathcal B(\mathbf p,D^j\mathbf p)$ for $j\in\Z$.
	Replacing $\mathbf p$ by $t\mathbf p$, and correspondingly
    $\mathbf q$ by $t\mathbf q$, where $t\in\mathbb R^\times$,
    replaces every $\Phi_j$ by $t^2\Phi_j$. Hence the signs of the
    $\Phi_j$ are independent of the chosen Pl\"ucker representative.
    We henceforth use the normalization specified in
    \cref{app:rank-four-algebra}.

	Since $\mathcal B$ is symmetric and $D$-invariant,
	\[
	\Phi_{-j}
	=\mathcal B(\mathbf p,D^{-j}\mathbf p)
	=\mathcal B(D^j\mathbf p,\mathbf p)
	=\mathcal B(\mathbf p,D^j\mathbf p)
	=\Phi_j.
	\]
	Thus $(\Phi_j)_{j\in\Z}$ is even. Finally,
	\eqref{eq:self-Phi}--\eqref{eq:cross-Phi} follow directly from
	$C=D^2$ and $\mathbf q=-D^{-1}\mathbf p$.
	
\end{proof}

 Since two planes in $\Gr_2(\R^4)$ are opposite if and only if the
 $\mathcal B$-pairing of their Pl\"ucker representatives is nonzero,
 \cref{lem:rank-four-Phi-reduction} reduces the opposition relations
involving $P,Q$ and their $c^k$-translates to the nonvanishing of
$\Phi_j$ for $j\neq0$. Thus it suffices to prove the following stronger statement.

 \begin{proposition}\label{prop:Phi-negative}
 	For the sequence in \cref{lem:rank-four-Phi-reduction}, one has
 	$\Phi_0=0$ and
 	\[
 	\Phi_j<0
 	\qquad
 	(j\in\Z\setminus\{0\}).
 	\]
 \end{proposition}

\subsection{Proof of Proposition~\ref{prop:Phi-negative}}

 \begin{lemma}\label{lem:Phi-recurrence-initial}
 	The sequence $(\Phi_j)_{j\in\Z}$ satisfies
 	\begin{equation}\label{eq:Phi-recurrence}
 		\Phi_{j+3}+\Phi_{j+2}-\Phi_{j+1}-3\Phi_j
 		-\Phi_{j-1}+\Phi_{j-2}+\Phi_{j-3}=0
 		\qquad(j\in\Z).
 	\end{equation}
	Moreover, for the normalization of $\mathbf p$ specified in
   \cref{app:rank-four-algebra}, one has $\Phi_0=0$ and
   \begin{equation}\label{eq:Phi-initial}
    \Phi_2=-\mathfrak a,
    \qquad
    \Phi_3=-\mathfrak b,
   \end{equation}
 
 	where
 	\begin{align*}
 		\mathfrak a
 		&=\frac{(29+12\sqrt5)\eta+744+408\sqrt5}{99},\\
 		\mathfrak b
 		&=\frac{(15\sqrt5-16)\eta-280+312\sqrt5}{99},
 	\end{align*}
 	with $\eta=\rho\sigma$. In particular,
 	$\mathfrak a,\mathfrak b>0$.
 \end{lemma}

 \begin{proof}
 	Set $\eta=\rho\sigma$. Since $P$ is spanned by the $\rho$- and
 	$\sigma$-eigenvectors of $z$, its Pl\"ucker line is an eigenline of
 	$\wedge^2z$ with eigenvalue $\eta$. A direct calculation gives
 	\[
 		\chi_{\wedge^2z}(T)
 		=(T-1)^2
 		\bigl(T^4+27T^3+173T^2+27T+1\bigr).
 	\]
 	Since $\eta\neq1$, it follows that
 	\[
 	\eta^4+27\eta^3+173\eta^2+27\eta+1=0.
 	\]
 	Evaluating the two quadratic factors in \eqref{eq:z4-charpoly} at
 	$12$ and $-1$, respectively, gives $\rho>12$ and $\sigma<-1$.
 	Thus $\eta<-12$, and hence $\eta+\eta^{-1}<-12$. Dividing the
 	preceding equation by $\eta^2$ gives
 	\[
 	(\eta+\eta^{-1})^2
 	+27(\eta+\eta^{-1})+171=0.
 	\]
 	Therefore
 	\begin{equation}\label{eq:eta-sqrt5}
 		\eta+\eta^{-1}
 		=-\frac{27+3\sqrt5}{2},
 	\end{equation}
 	and in particular $\eta>-17$.
 	
 		Let $\lambda_1,\lambda_2,\lambda_3$ be the roots of $T^3-T-1$.
 	By \eqref{eq:d-square}, the eigenvalues of $d$ are
 	$1,\lambda_1,\lambda_2,\lambda_3$, with
 	$\lambda_1\lambda_2\lambda_3=1$. Hence the eigenvalues of
 	$D=\wedge^2d$ are
 	\[
 	\lambda_1,\lambda_2,\lambda_3,
 	\lambda_1^{-1},\lambda_2^{-1},\lambda_3^{-1}.
 	\]
 	Consequently,
 	\[
 		\chi_D(T)
 		=(T^3-T-1)(T^3+T^2-1)
 		=T^6+T^5-T^4-3T^3-T^2+T+1.
 	\]
 	By Cayley--Hamilton,
 	\[
 	D^6+D^5-D^4-3D^3-D^2+D+I=0.
 	\]
 	Multiplying by $D^{j-3}$ and pairing with $\mathbf p$ gives
 	\eqref{eq:Phi-recurrence}.
 	
 	Since $\mathbf p$ is decomposable,
 	$\Phi_0=\mathcal B(\mathbf p,\mathbf p)=0$.
 	With the normalization of $\mathbf p$ used in
 	\cref{app:rank-four-algebra}, the coordinate calculation there gives
 	\eqref{eq:Phi-initial}.
 	
 	Finally, since $\eta>-17$,
 	\[
 	99\mathfrak a>251+204\sqrt5>0,
 	\qquad
 	99\mathfrak b>57\sqrt5-8>0.
 	\]
 	This proves the lemma.
 	\end{proof}

  \begin{lemma}\label{lem:bilateral-padovan}
  	Let $(\omega_j)_{j\in\Z}$ be defined by
  	$\omega_0=0$, $\omega_1=1$, $\omega_2=0$, and
  	$\omega_{j+3}=\omega_{j+1}+\omega_j$.
  	For $j\geq1$, put
  	$\alpha_j=\omega_j+\omega_{-j}$ and
  	$\beta_j=\omega_{j-1}+\omega_{-j-1}$.
  	Then
  	\[
  	\alpha_j,\beta_j\geq0,
  	\qquad
  	\alpha_j+\beta_j>0
  	\qquad(j\geq1).
  	\]
  \end{lemma}

 \begin{proof}
 	
 	Let $\varrho>1$ be the real root of $X^3-X-1$, and let
 	$\zeta,\bar\zeta$ be the remaining two roots. Then
 	$|\zeta|=\varrho^{-1/2}$. For
 	$F(X):=\sum_{j\ge0}\omega_jX^j$, the recurrence gives
 	\[
 	F(X)=\frac{X}{1-X^2-X^3}
 	=\frac{X}{(1-\varrho X)(1-\zeta X)(1-\bar\zeta X)}.
 	\]
 	Hence
 	\[
 	F(X)
 	=\frac{\gamma}{1-\varrho X}
 	+\frac{\delta}{1-\zeta X}
 	+\frac{\bar\delta}{1-\bar\zeta X},
 	\]
 	where $\gamma=\frac{\varrho^2}{2\varrho+3}$ and
 	$\delta=\frac{\zeta^2}{2\zeta+3}$. Expanding each term as a geometric
 	series and comparing coefficients gives
 	\begin{equation}\label{eq:omega-binet}
 		\omega_j
 		=\gamma\varrho^j+2\operatorname{Re}(\delta\zeta^j)
 		\qquad(j\geq0).
 	\end{equation}
 	Since both sides satisfy the same bilateral recurrence, this formula
 	extends to every $j\in\Z$.

 	Let $f(X)=X^3-X-1$. Since
 	$f(13/10)<0$ and $f(4/3)>0$, we have
 	$13/10<\varrho<4/3$.
 	Since $\gamma=\varrho^2/(2\varrho+3)$ and the function
 	$x\mapsto x^2/(2x+3)$ is increasing for $x>0$,
 	$
 	\gamma>
 	(13/10)^2/(2(13/10)+3)
 	=169/560
 	>3/10.
 	$
	Moreover, $|\delta|^2=|\zeta|^4/|2\zeta+3|^2$.
 	By Vieta's formulas, $\zeta+\bar\zeta=-\varrho$ and
 	$\zeta\bar\zeta=\varrho^{-1}$, and hence
 	$
 	|2\zeta+3|^2
 	=(2\zeta+3)(2\bar\zeta+3)
 	=\frac4\varrho-6\varrho+9.
 	$
 	Since $|\zeta|^2=\varrho^{-1}$ and $\varrho^3=\varrho+1$, it follows that
 	$
 	|\delta|^2
 	=1/(9\varrho^2-2\varrho-6)
 	<\frac4{25},
 	$
 	where the last inequality follows from $\varrho>13/10$. Thus
 	$|\delta|<2/5$.
 	Now, by the formula for $\omega_j$,
 	\[
 		\alpha_j=\omega_j+\omega_{-j}=\gamma(\varrho^j+\varrho^{-j})
 		+2\operatorname{Re}
 		\bigl(\delta\zeta^j+\delta\zeta^{-j}\bigr).
 	\]
 	Using $\operatorname{Re}(w)\geq-|w|$ and
 	$|\zeta|=\varrho^{-1/2}$, we obtain
 	\[
 		\alpha_j\geq
 		\gamma(\varrho^j+\varrho^{-j})
 		-2|\delta|
 		\bigl(\varrho^{-j/2}+\varrho^{j/2}\bigr)
 		>
 		\varrho^{j/2}
 		\bigl(\gamma\varrho^{j/2}-4|\delta|\bigr).
 	\]
 	Similarly, $
 		\beta_j=\omega_{j-1}+\omega_{-j-1}
 		=\gamma(\varrho^{j-1}+\varrho^{-j-1})
 		+2\operatorname{Re}
 		\bigl(\delta\zeta^{j-1}
 		+\delta\zeta^{-j-1}\bigr)$,
 	and therefore $\beta_j>\varrho^{(j+1)/2}
 		\bigl(\gamma\varrho^{(j-3)/2}-4|\delta|\bigr)$.
 	
 	Finally,
 	$4|\delta|/\gamma< (4(2/5))/(3/10)=16/3$ and $\varrho^7>(13/10)^7>16/3$.
 	Hence, if $j\geq17$, then
 	$\varrho^{j/2}>4|\delta|/\gamma$ and
 	$\varrho^{(j-3)/2}>4|\delta|/\gamma$, so
 	$\alpha_j>0$ and $\beta_j>0$.

 	For $1\leq j\leq16$, direct iteration of the recurrence gives
 	$\alpha_j,\beta_j\geq0$; the only zeros are
 	$\alpha_3=0$ and $\beta_2=\beta_5=\beta_7=0$.
 	Thus $\alpha_j+\beta_j>0$ for every $j\geq1$.
 \end{proof}

\begin{proof}[Proof of \cref{prop:Phi-negative}]
	Extend $\alpha_j$ and $\beta_j$ to all $j\in\Z$ by the same formulas.
	Then, for every $j\in\Z$,
	$\alpha_{-j}=\alpha_j$ and $\beta_{-j}=\beta_j$.
	Let $T$ denote the shift operator on bilateral sequences,
	$(Tu)_j=u_{j+1}$. The recurrence
	$\omega_{j+3}=\omega_{j+1}+\omega_j$ is equivalent to
	$
	(T^3-T-1)\omega=0.
	$
	If $\widetilde\omega_j:=\omega_{-j}$, then
	$\widetilde\omega_{j+3}+\widetilde\omega_{j+2}-\widetilde\omega_j=0$,
	so
	$
	(T^3+T^2-1)\widetilde\omega=0.
	$
	Since
	$\alpha=\omega+\widetilde\omega$, $\beta=T^{-1}\omega+T\widetilde\omega$,
	and the shift operator commutes with every polynomial in $T$, both
	$\alpha$ and $\beta$ are annihilated by
	$
	(T^3-T-1)(T^3+T^2-1).
	$
	Equivalently, both sequences satisfy \eqref{eq:Phi-recurrence}.
	Moreover,
	\[
	(\alpha_0,\beta_0)=(0,0),
	\qquad
	(\alpha_2,\beta_2)=(1,0),
	\qquad
	(\alpha_3,\beta_3)=(0,1).
	\]

	For an even sequence $(x_j)_{j\in\Z}$ satisfying
	\eqref{eq:Phi-recurrence} with $x_0=0$, the case $j=0$ gives
	$x_1=x_2+x_3$. Hence such a sequence is determined by $x_2$ and
	$x_3$. By \cref{lem:Phi-recurrence-initial}, it follows that
	\[
	\Phi_j
	=-\mathfrak a\,\alpha_j-\mathfrak b\,\beta_j
	\qquad(j\in\Z).
	\]
	For $j\geq1$, \cref{lem:bilateral-padovan} and
	$\mathfrak a,\mathfrak b>0$ imply $\Phi_j<0$. Since $(\Phi_j)$ is
	even, the same holds for $j\leq-1$, while $\Phi_0=0$. This proves the
	proposition.
\end{proof}

\subsection{Rank-four admissibility conditions}

We now translate the information about $\Phi_j$ into opposition
relations in $\Gr_2(\R^4)$.

\begin{lemma}\label{lem:grassmann-separation}
For the rank-four pair \eqref{eq:p4-q4}, the following hold.
\begin{enumerate}[label=\textup{(\roman*)}]
 \item For every $k\in\Z$,
 \[
  \mathcal B(\mathbf p,C^k\mathbf q)>0,
  \qquad
  \mathcal B(\mathbf q,C^k\mathbf p)>0.
 \]
 \item For every $k\neq0$,
 \[
  \mathcal B(\mathbf p,C^k\mathbf p)<0,
  \qquad
  \mathcal B(\mathbf q,C^k\mathbf q)<0.
 \]
 \item The planes $P,Q$ are opposite to each other and to both fixed planes $c^+,c^-$.  Moreover, for every $k\neq0$ and every $R_0,S_0\in\{P,Q\}$, the planes $R_0$ and $c^kS_0$ are opposite.
\end{enumerate}
\end{lemma}

\begin{proof}
 Parts (i)--(ii) follow immediately from \eqref{eq:self-Phi}, \eqref{eq:cross-Phi}, and \cref{prop:Phi-negative}.  In particular, $P$ and $Q$ are opposite.

 It remains to prove that $P$ and $Q$ are opposite to $c^+$ and $c^-$. 
 By \eqref{eq:omega-binet},
 \[
 \lim_{j\to\infty}\varrho^{-j}\Phi_j
 =-\gamma\left(\mathfrak a+\frac{\mathfrak b}{\varrho}\right)\neq0.
 \]
 By \eqref{eq:d-square},
 $
\chi_d(T)=(T-1)(T^3-T-1).
 $
The cubic polynomial $T^3-T-1$ has discriminant $-23$ and does
not vanish at $1$. Hence $\chi_d$ has four distinct roots, so
$d$ is diagonalizable over $\mathbb C$. Consequently,
$D=\wedge^2d$ is also diagonalizable over $\mathbb C$.
 Let $E_+$ and $E_-$ be the one-dimensional eigenspaces of
 $D=\wedge^2d$ corresponding to the eigenvalues $\varrho$ and
 $\varrho^{-1}$, respectively, and choose nonzero vectors
 $v_+\in E_+$ and $v_-\in E_-$.  Write
 $
 \mathbf p=a_+v_++a_-v_-+\mathbf p_0,
 $
 where $\mathbf p_0$ belongs to the sum of the remaining eigenspaces of
 $D$.  Since $D$ preserves $\mathcal B$, if
 $u\in E_\lambda$ and $v\in E_\mu$, then
 \[
 \mathcal B(u,v)
 =\mathcal B(Du,Dv)
 =\lambda\mu\,\mathcal B(u,v).
 \]
 Hence $\mathcal B(E_\lambda,E_\mu)=0$ unless
 $\lambda\mu=1$.  In particular,
 $\mathcal B(v_-,v_+)\neq0$, since $\mathcal B$ is nondegenerate.
 
 The remaining eigenvalues of $D$ have modulus at most
 $\varrho^{1/2}$.  Therefore
 \[
 \Phi_j
 =\mathcal B(\mathbf p,D^j\mathbf p)
 =a_+a_-\,\mathcal B(v_-,v_+)\varrho^j
 +O(\varrho^{j/2}),
 \]
 and hence
 \[
 a_+a_-\,\mathcal B(v_-,v_+)
 =\lim_{j\to\infty}\varrho^{-j}\Phi_j
 =-\gamma\left(\mathfrak a+\frac{\mathfrak b}{\varrho}\right)
 \neq0.
 \]
 Thus $a_+\neq0$ and $a_-\neq0$.  Moreover,
 $\mathcal B(\mathbf p,v_+)
 =a_-\mathcal B(v_-,v_+)\neq0$ and
 $\mathcal B(\mathbf p,v_-)
 =a_+\mathcal B(v_+,v_-)\neq0$.
 The Pl\"ucker lines of the attracting and repelling fixed planes
 $c^+$ and $c^-$ of $c=d^2$ are $E_+$ and $E_-$, respectively.
 \eqref{eq:grassmann-opposition} shows that $P$ is opposite to both
 $c^+$ and $c^-$.
 
 A representative of $Q=d^{-1}P$ is $D^{-1}\mathbf p$, and
 \[
 D^{-1}\mathbf p
 =a_+\varrho^{-1}v_+
 +a_-\varrho v_-
 +D^{-1}\mathbf p_0.
 \]
 Thus its components in $E_+$ and $E_-$ are again nonzero, and the same
 argument shows that $Q$ is opposite to both $c^+$ and $c^-$.  Finally,
 \eqref{eq:grassmann-opposition} together with (i)--(ii) gives that,
 for every $k\neq0$ and $R_0,S_0\in\{P,Q\}$, the planes $R_0$ and
 $c^kS_0$ are opposite.

\end{proof}

\begin{proposition}\label{prop:rank4-admissible}
The pair $(p,q)$ in \eqref{eq:p4-q4} generates $\SL_4(\Z)$ and is $\Gr_2(\R^4)$-admissible.
\end{proposition}

\begin{proof}
 By \cref{lem:rank-four-generating-pair}, the pair $(p,q)$ generates
 $\SL_4(\Z)$ and $z=pq$ is biproximal on $\Gr_2(\R^4)$.
 By \cref{lem:c4-biproximal}, $c$ is also biproximal on
 $\Gr_2(\R^4)$.  The remaining admissibility conditions
 \textup{(2)}--\textup{(3)} follow from
 \cref{lem:grassmann-separation}.
\end{proof}

\begin{corollary}\label{cor:rank4-free-anosov}
	There exists $L_0$ such that, for every $L\geq L_0$,
	the group $\langle z^Lp,c\rangle$ is free and $P_2$-Anosov.
\end{corollary}

\begin{proof}
	This follows immediately from
	\cref{prop:rank4-admissible,prop:admissible-criterion}.
\end{proof}

\section{Construction of free subgroups of rank $m$}\label{sec:free-corridors}

We now construct the subgroups $H_{N,m}$ used in
\cref{lem:corridor} and prove that they are free of rank $m$ for all
sufficiently large $N$.

 Fix $n\geq3$. If $n\neq4$, use the Jordan pair from
 \cref{sec:jordan-family}; if $n=4$, use the pair from
 \cref{sec:rank-four}. Choose $L$ sufficiently large, and write
 $\Gamma=\SL_n(\Z)$, $a=a_L$, $b=b_L$, and
 $F=\langle a,c\rangle$. By the preceding sections,
 $\langle a,b\rangle=\Gamma$, while $F$ is an Anosov free group with
 free basis $(a,c)$.
 
 Fix $m\geq3$. For $N\geq1$, set
 \[
 x_N=a^Nb,\qquad
 t=a^{-1},\qquad
 h_j=t^jx_Nt^{-j}=a^{N-j}ba^j,
 \]
 and let $H_{N,m}=\langle h_0,\ldots,h_{m-1}\rangle$.
 
   \begin{lemma}\label{lem:corridor-algebra}
   	For every $N\geq1$, the following hold.
   	\begin{enumerate}[label=\textup{(\roman*)}]
   		\item $\langle x_N,t\rangle=\Gamma$.
   		
   		\item If $c_j=a^{-j}ca^j$ and
   		$J_m=\langle c_0,\ldots,c_{m-2}\rangle$, then
   		$H_{N,m}=\langle J_m,x_N\rangle$.
   		
   		\item The group $J_m$ is free of rank $m-1$. In particular,
   		$J_m$ is non-elementary and quasiconvex in $F$.
   	\end{enumerate}
   \end{lemma}

 \begin{proof}
 	Since $a=t^{-1}$ and $b=a^{-N}x_N$, we have
 	$\langle x_N,t\rangle=\Gamma$, proving \textup{(i)}.
 	
 	A direct calculation gives
 	$
 	h_j^{-1}h_{j+1}=a^{-j}ca^j=c_j.
 	$
 	Hence $J_m\leq H_{N,m}$. Conversely, since $h_0=x_N$ and
 	$h_{j+1}=h_jc_j$, every $h_j$ belongs to
 	$\langle J_m,x_N\rangle$, proving \textup{(ii)}.
 	
 	Let $\epsilon_F:F=\langle a,c\rangle\to\Z$ be defined by
 	$\epsilon_F(a)=1$ and $\epsilon_F(c)=0$. Using the Schreier
 	transversal $\{a^j:j\in\Z\}$ for $\ker\epsilon_F$ in $F$,
 	Schreier's method gives
 	$
 	\ker\epsilon_F
 	=\left\langle a^{-j}ca^j:j\in\Z\right\rangle,
 	$
 	with $\{a^{-j}ca^j:j\in\Z\}$ as a free basis; see
 	\cite[Chapter~I, Proposition~3.7]{LyndonSchupp1977}. Thus
 	$c_0,\ldots,c_{m-2}$ freely generate $J_m$, so
 	$J_m\cong F_{m-1}$. Since $m\geq3$, $J_m$ is non-elementary; being a
 	finitely generated subgroup of the free group $F$, it is quasiconvex.
 \end{proof}

To show that $\langle J_m,a^Nb\rangle$ is a free product, we verify
the boundary-position hypotheses of \cref{prop:relative-pingpong}.
Define $w=ca^{-1}=b^{-1}a^{-1}b\in F$, and let
$p_\infty=a_F^+\in\partial F$ and $q_\infty=w_F^+\in\partial F$.

\begin{lemma}\label{lem:boundary-position}
The points $p_\infty,q_\infty$ lie outside $\partial J_m$.  Their stabilizers in $J_m$ are trivial, and
\[
 J_mp_\infty\cap J_mq_\infty=\varnothing.
\]
\end{lemma}

\begin{proof}

    We identify $\partial J_m$ with its limit set
    $\Lambda_F(J_m)\subseteq\partial F$.
    The subgroups $J_m$, $\langle a\rangle$, and
    $\langle w\rangle$ are quasiconvex in $F$. By the limit-set
    intersection theorem
    \cite[Lemma~2.6]{GitikMitraRipsSageev1998},
    \[\Lambda_F(J_m)\cap\Lambda_F(\langle a\rangle)
    =\Lambda_F(J_m\cap\langle a\rangle).\]
  	Since $J_m\subseteq\ker\epsilon_F$ and $\epsilon_F(a^k)=k$, we have
  	$J_m\cap\langle a\rangle=\{1\}$. Hence
  	$\Lambda_F(J_m)\cap\Lambda_F(\langle a\rangle)=\varnothing$, and so $p_\infty=a_F^+\notin\partial J_m$.

  	Similarly, since $\epsilon_F(w^k)=-k$, we have
  	$J_m\cap\langle w\rangle=\{1\}$. Therefore
  	$\Lambda_F(J_m)\cap\Lambda_F(\langle w\rangle)=\varnothing$, and hence $q_\infty=w_F^+\notin\partial J_m$. 	
  	
  	If a nontrivial element $u\in J_m$ fixed $p_\infty$, then
  	$p_\infty\in\Lambda_F(\langle u\rangle)\subseteq\Lambda_F(J_m)$,
  	a contradiction. Thus the stabilizer of $p_\infty$ in $J_m$ is
  	trivial. The same argument applies to $q_\infty$.
  	
  	Finally, suppose that
  	$J_mp_\infty\cap J_mq_\infty\neq\varnothing$. Then there exists
  	$h\in J_m$ such that $hp_\infty=q_\infty$. Since
  	$p_\infty=a_F^+$ and $q_\infty=w_F^+$, we have $(hah^{-1})_F^+=hp_\infty=q_\infty=w_F^+$.
  	Thus $hah^{-1}$ and $w$ have the same attracting fixed point
    $q_\infty$. In the Cayley tree of $F$, their axes therefore share
    a ray. Choosing positive integers $r,s$ so that the translation
    lengths of $(hah^{-1})^r$ and $w^s$ agree, these two elements act
    identically on that ray. Since the action of $F$ on its Cayley tree
    is free, it follows that
    $
     (hah^{-1})^r=w^s.
    $
	Now $h\in J_m\subseteq\ker\epsilon_F$, so
  	$\epsilon_F(hah^{-1})=1$, while $\epsilon_F(w)=-1$. Applying
  	$\epsilon_F$ gives $r=-s$, a contradiction. Hence
  	$J_mp_\infty\cap J_mq_\infty=\varnothing$.

\end{proof}

\begin{proposition}\label{prop:free-product-corridor}
For every $m\geq3$, there exists $N_0(n,L,m)$ such that, for every $N\geq N_0(n,L,m)$,
\[
 H_{N,m}=\langle J_m,a^Nb\rangle
 \cong J_m*\langle a^Nb\rangle
 \cong F_m.
\]
\end{proposition}

\begin{proof}
	
		Let $r_\infty=a_F^-\in\partial F$. Since the Anosov boundary map
		$\zeta:\partial F\to\cX$ is dynamics preserving,
		$\zeta(p_\infty)$ and $\zeta(r_\infty)$ are the attracting flags of
		$a$ and $a^{-1}$, respectively. Since
		$w=b^{-1}a^{-1}b$, we also have
		$\zeta(q_\infty)=b^{-1}\zeta(r_\infty)$.
		
		Apply \cref{lem:regular-ray-general} with $z=a$, $u=b$, and
		$g_N=a^Nb$. Then
		\[
		a^Nb|_{\Opp(\zeta(q_\infty))}
		\longrightarrow \zeta(p_\infty),
		\qquad
		b^{-1}a^{-N}|_{\Opp(\zeta(p_\infty))}
		\longrightarrow \zeta(q_\infty)
		\]
		locally uniformly.
		
		By \cref{lem:corridor-algebra}\textup{(iii)}, $J_m$ is
		non-elementary and quasiconvex in $F$, while
		\cref{lem:boundary-position} gives
		$p_\infty,q_\infty\notin\partial J_m$, trivial stabilizers in $J_m$,
		and
		$J_mp_\infty\cap J_mq_\infty=\varnothing$.
		Thus all the hypotheses of \cref{prop:relative-pingpong} are
		satisfied. Hence, for all sufficiently large $N$,
		$
		\langle J_m,a^Nb\rangle
		\cong J_m*\langle a^Nb\rangle,
		$
		and $a^Nb$ has infinite order. Therefore
		$\langle a^Nb\rangle\cong\Z$. Together with
		\cref{lem:corridor-algebra}\textup{(ii)--(iii)}, this gives
		$
		H_{N,m}
		\cong F_{m-1}*\Z
		\cong F_m.
		$

\end{proof}

\section{Infinite index and distinct cosets}\label{sec:cosets}
By \cref{prop:free-product-corridor}, the subgroup $H_{N,m}$ is free
of rank $m$ for all sufficiently large $N$. To apply
\cref{lem:corridor}, it remains to prove that $H_{N,m}$ has infinite
index and that the relevant cosets are pairwise distinct.

\begin{corollary}\label{cor:infinite-index-general}
For $m\geq3$ and $N\geq N_0(n,L,m)$,
$
 [\Gamma:H_{N,m}]=\infty.
$
\end{corollary}

\begin{proof}
By \cref{prop:free-product-corridor}, $H_{N,m}\cong F_m$.  If it had finite index in $\Gamma=\SL_n(\Z)$, it would have property $(T)$, since $\SL_n(\Z)$ has property $(T)$ for $n\geq3$ and property $(T)$ passes to finite-index subgroups; see
\cite[Theorem~1.7.1 and Example~1.7.4(i)]{BHV2008}.  But $F_m$ has an infinite cyclic quotient, and property $(T)$ passes to quotients.  This is impossible.
\end{proof}

\begin{lemma}\label{lem:normal-closure-general}
For $N\geq1$, let $K_N=\langle h_j:j\in\Z\rangle$. Then $K_N=\Gamma$.
\end{lemma}

\begin{proof}
The relation $a^{-1}h_ja=h_{j+1}$ shows that $a$ normalizes $K_N$.  Since $h_0=a^Nb\in K_N$ and $a^{-N}$ normalizes $K_N$, the identity $b=a^{-N}h_0$ shows that $b$ also normalizes $K_N$.  Hence
$
 K_N\triangleleft\langle a,b\rangle=\Gamma.
$
In $\Gamma/K_N$, the relation $a^Nb=1$ holds, so the quotient is cyclic.  The group $\SL_n(\Z)$ is perfect for $n\geq3$: it is generated by the elementary matrices $e_{ij}(1)$, and, for distinct $i,j,k$,
$
 e_{ij}(1)=[e_{ik}(1),e_{kj}(1)].
$
Thus $\Gamma$ has no nontrivial cyclic quotient, and $K_N=\Gamma$.
\end{proof}

\begin{proposition}\label{prop:distinct-cosets-general}
	For $m\geq3$ and $N\geq N_0(n,L,m)$, one has
	$a^r\notin H_{N,m}$ for every $1\leq r\leq m$.
	Consequently, the cosets
	$H_{N,m},H_{N,m}t,\ldots,H_{N,m}t^m$
	are pairwise distinct.
\end{proposition}

\begin{proof}
		Suppose that $a^r\in H_{N,m}$ for some $1\leq r\leq m$.
		Then $h_0,\ldots,h_{r-1}\in H_{N,m}$. Since
		$a^{qr}\in H_{N,m}$ for every $q\in\Z$, conjugation by $a^{qr}$
		preserves $H_{N,m}$, and hence
		\[
		a^{-qr}h_sa^{qr}=h_{s+qr}\in H_{N,m}
		\qquad(0\leq s<r,\ q\in\Z).
		\]
		Every $j\in\Z$ can be written as $j=s+qr$ with
		$0\leq s<r$, so $h_j\in H_{N,m}$ for every $j\in\Z$.
		Thus $K_N\leq H_{N,m}$, and
		\cref{lem:normal-closure-general} gives
		$\Gamma=K_N\leq H_{N,m}$, contradicting
		\cref{cor:infinite-index-general}.
		
		Now suppose that $0\leq i<j\leq m$ and
		$H_{N,m}t^i=H_{N,m}t^j$. Then
		$
		t^{i-j}=a^{j-i}\in H_{N,m}.
		$
		Since $1\leq j-i\leq m$, this contradicts the first part.
	
\end{proof}

We have now verified all the hypotheses of \cref{lem:corridor}.

 \begin{corollary}\label{cor:final-bound}

    For every $m\geq3$, there exists $N_0=N_0(n,L,m)$ such that,
   for every $N\geq N_0$,
   $
   \kappa\bigl(\Gamma,\{a^Nb,a^{-1}\}\bigr)
   \leq\sqrt{2/m}.
   $
   Consequently,
$\kappa\bigl(\Gamma,\{a^Nb,a^{-1}\}\bigr)\to0$ as $N\to\infty$.
 \end{corollary}
 
 \begin{proof}
 	This follows from
 	\cref{lem:corridor,prop:free-product-corridor,cor:infinite-index-general,prop:distinct-cosets-general}.
 \end{proof}

\section{Proof of the main theorem}

\begin{proof}[Proof of \cref{thm:main}]
Fix $n\geq3$.  Choose the admissible generating pair from \cref{sec:jordan-family} when $n\neq4$, and from \cref{sec:rank-four} when $n=4$.  Choose $L$ large enough that $F=\langle a_L,c\rangle$ is Anosov and free.  For each $m\geq3$, choose $N\geq N_0(n,L,m)$.  By \cref{lem:corridor-algebra}\textup{(i)},
$
 S_{N,m}=\{a_L^Nb_L,a_L^{-1}\}
$
is a two-element generating set of $\SL_n(\Z)$.  By \cref{cor:final-bound},
$
 \kappa\bigl(\SL_n(\Z),S_{N,m}\bigr)
 \leq\sqrt{2/m}.
$
Since $m$ is arbitrary,
$
 0\leq\kappa_{\leq2}\bigl(\SL_n(\Z)\bigr)
 \leq\inf_{m\geq3}\sqrt{2/m}=0.
$
This proves the theorem.
\end{proof}

\appendix

\section{The alternating sine sums}\label{app:sine-sums}

We prove the identities used in \eqref{eq:alternating-sums}.  Put $\theta=\pi/(2n+1)$ and
\[
 A_j=(-1)^{j-1}\sin(2j\theta),
 \qquad
 r=-e^{2i\theta}.
\]
Then $A_j=-\operatorname{Im}(r^j)$.  The finite geometric sum gives
\[
 \sum_{j=1}^n(-1)^{j-1}\sin(2j\theta)
 =\frac{\sin\theta-(-1)^n\sin((2n+1)\theta)}{2\cos\theta}
 =\frac{\sin\theta}{2\cos\theta}.
\]
For the weighted sum, use
\[
 \sum_{j=1}^n(n+1-j)r^j
 =\frac{r\bigl(n-(n+1)r+r^{n+1}\bigr)}{(1-r)^2}.
\]
Since $1-r=2\cos\theta\,e^{i\theta}$, one has
$r/(1-r)^2=-1/(4\cos^2\theta)$.  Moreover,
$2(n+1)\theta=\pi+\theta$ and, with $\eps=(-1)^{n-1}$,
\[
 \operatorname{Im}\bigl(-(n+1)r+r^{n+1}\bigr)
 =2(n+1)\sin\theta\cos\theta-\eps\sin\theta.
\]
Taking minus the imaginary part therefore yields
\[
 \sum_{j=1}^n(n+1-j)A_j
 =\frac{\sin\theta\bigl(2(n+1)\cos\theta-\eps\bigr)}
        {4\cos^2\theta},
\]
as claimed.

\section{The rank-four Pl\"ucker calculation}
\label{app:rank-four-algebra}

  We verify the values of $\Phi_2$ and $\Phi_3$ in
  \eqref{eq:Phi-initial} for a convenient normalization of the
  Pl\"ucker representative $\mathbf p$. Put
$\xi=\rho+\sigma$, and $\eta=\rho\sigma$.
Comparison of coefficients in \eqref{eq:z4-charpoly} gives
$\xi=11\eta/(\eta+1)$.

For a root $t$ of $\chi_z$, set
\[
v(t)=
\begin{pmatrix}
	-5t^2-4t\\
	-t^3+t^2\\
	t^2-t\\
	-5t^2-5t+1
\end{pmatrix}.
\]
A direct calculation gives
\[
(z-tI)v(t)=
\begin{pmatrix}
	0\\
	\chi_z(t)\\
	0\\
	0
\end{pmatrix},
\]
so $v(\rho)$ and $v(\sigma)$ are eigenvectors corresponding to
$\rho$ and $\sigma$, respectively. We normalize $\mathbf p$ by
\[
\mathbf p
=-\frac{v(\rho)\wedge v(\sigma)}
{9\eta(\rho-\sigma)}.
\]
In the ordered Pl\"ucker basis
$
e_{12},e_{13},e_{14},e_{23},e_{24},e_{34},
$
this gives
\[
\mathbf p=
\begin{pmatrix}
	\dfrac{5\eta+4\xi-4}{9}\\[1mm]
	-1\\[1mm]
	\dfrac{-5\eta+5\xi+4}{9\eta}\\[1mm]
	-\dfrac{\xi-\eta-1}{9}\\[1mm]
	\dfrac{\xi^2-5\xi\eta-\xi-5\eta^2+4\eta}{9\eta}\\[1mm]
	-\dfrac{\xi-10\eta-1}{9\eta}
\end{pmatrix}.
\]

In the same basis,
\[
D=\wedge^2d=
\begin{pmatrix}
	1&-3&-1&1&1&-2\\
	0&-1&0&1&0&0\\
	-1&0&0&2&0&0\\
	0&-2&0&1&0&-1\\
	-3&3&1&0&-2&2\\
	0&1&0&-2&0&0
\end{pmatrix},
\]
while the matrix of $\mathcal B$ is
\[
J=
\begin{pmatrix}
	0&0&0&0&0&1\\
	0&0&0&0&-1&0\\
	0&0&0&1&0&0\\
	0&0&1&0&0&0\\
	0&-1&0&0&0&0\\
	1&0&0&0&0&0
\end{pmatrix}.
\]
Thus $\mathcal B(x,y)=x^{\transp}Jy$ in these coordinates, and hence
$
\Phi_j=\mathbf p^{\transp}JD^j\mathbf p.
$

Using $\xi=11\eta/(\eta+1)$ and, by \eqref{eq:eta-sqrt5},
$
\eta^2+\frac{27+3\sqrt5}{2}\eta+1=0,
$
direct simplification gives
\[
\Phi_2
=-\frac{(29+12\sqrt5)\eta+744+408\sqrt5}{99}
=-\mathfrak a
\]
and
\[
\Phi_3
=-\frac{(15\sqrt5-16)\eta-280+312\sqrt5}{99}
=-\mathfrak b.
\]
This proves \eqref{eq:Phi-initial}.

\section*{Acknowledgments}
GPT-5.6 assisted in the development of the proofs through an iterative
dialogue with the author. The author takes full responsibility for all
mathematical arguments and for the content of the paper. The author 
thanks Alexander Lubotzky for helpful discussions and valuable suggestions.

\end{document}